\documentclass[a4paper]{amsart}

\usepackage[utf8x]{inputenc}
\usepackage{enumerate}
\usepackage{hyperref}
\usepackage[arrow, matrix]{xy}
\usepackage{amssymb}

\newtheorem{maintheorem}{Theorem}
\newtheorem{theorem}{Theorem}
\newtheorem{lemma}[theorem]{Lemma}
\newtheorem{proposition}[theorem]{Proposition}
\newtheorem{corollary}[theorem]{Corollary}

\newcommand{\N}{\mathbb{N}}
\newcommand{\R}{\mathbb{R}}

\newcommand{\Q}{\mathbb{Q}}
\newcommand{\C}{\mathbb{C}}
\renewcommand{\P}{\mathbb{P}}
\newcommand{\Norm}{\mathfrak{N}}
\newcommand{\p}{\mathfrak{p}}
\renewcommand{\O}{\mathcal{O}}
\newcommand{\ida}{\mathfrak{a}}
\newcommand{\idb}{\mathfrak{b}}
\newcommand{\idc}{\mathfrak{c}}

\DeclareMathOperator{\ideal}{\mathfrak{J}}
\DeclareMathOperator{\Vol}{Vol}
\DeclareMathOperator{\Lip}{Lip}

\title[Rational points over number fields on a singular cubic surface]{Counting rational points over number fields on a singular cubic surface}
\author{Christopher Frei}
\address{Mathematisches Institut, LMU M\"unchen, Theresienstr. 39, 80333 M\"unchen, Germany}
\email{frei@math.lmu.de}
\urladdr{http://www.mathematik.uni-muenchen.de/~frei}

\begin{document}
\numberwithin{equation}{section}
\numberwithin{theorem}{section}
\begin{abstract}
A conjecture of Manin predicts the distribution of $K$-rational points on certain algebraic varieties defined over a number field $K$. In recent years, a method using universal torsors has been successfully applied to several hard special cases of Manin's conjecture over the field $\Q$. Combining this method with techniques developed by Schanuel, we give a proof of Manin's conjecture over arbitrary number fields for the singular cubic surface $S$ given by the equation $x_0^3 = x_1 x_2 x_3$. 
\end{abstract}

\maketitle
\tableofcontents

\section{Introduction}
We consider the cubic surface $S \subseteq \P^3$ defined over any number field $K$ by the equation
\[x_0^3 = x_1 x_2 x_3\text.\]
It is toric, has three singular points $(0:1:0:0)$, $(0:0:1:0)$, $(0:0:0:1)$, and contains three lines $L_i := \{x_0 = x_i = 0\}$, for $i \in \{1, 2, 3\}$. The set $S(K)$ of $K$-rational points on $S$ is infinite.

The Weil height of ${\bf x} = (x_0 : x_1 : x_2 : x_3) \in \P^3(K)$ is defined by 
\[H({\bf x}) = \prod_{\nu \in M(K)} \max\{|x_0|_\nu, |x_1|_\nu, |x_2|_\nu, |x_3|_\nu\}^{d_\nu}\text.\]
Here, $M(K)$ is the set of places of $K$, the absolute values $|\cdot |_\nu$ are normalized such that they extend the usual absolute values on $\Q$, and $d_\nu$ is the local degree $[K_\nu : \Q_p]$, if $\nu$ extends the place $p$ of $\Q$. 

It is well known that there are only finitely many points of bounded height in $\P^3(K)$, so it makes sense to study the number of $K$-rational points on $S$ of height bounded by $B$, as $B$ tends to infinity. A generalization of a conjecture by Manin \cite{Franke1989, Batyrev1998b}, applied to our case, links the asymptotic behavior of this quantity to geometric features of $S$, provided that we exclude the points lying on the lines $L_i$. Indeed, the number of $K$-rational points of bounded height on these lines dominates the number of $K$-rational points on the rest of $S$, whereas much of the geometric information about $S$ would be lost when considering just the lines.

Therefore, we denote by $U$ the complement of the three lines in $S$ and define the counting function
\[N(B) := |\{{\bf x} \in U(K) \mid H({\bf x}) \leq B\}|\text.\]
Here, $U(K)$ is the set of $K$-rational points on $U$. The above-mentioned generalization of Manin's conjecture \cite{Franke1989, Batyrev1998b} to Fano varieties with at worst canonical singularities predicts in this case that
\[N(B) \sim c B (\log B)^6\text,\]
with a positive leading constant $c = c_{S, K, H}$. A conjectural interpretation of the leading constant in Manin's conjecture was given by Peyre \cite{Peyre1995} and extended to Fano varieties with at worst canonical singularities by Batyrev and Tschinkel \cite{Batyrev1998b}. When writing ``Manin's conjecture'', we implicitly include the conjecture about the leading constant.

Manin's conjecture has been proved for smooth toric varieties over arbitrary number fields by Batyrev and Tschinkel \cite{Batyrev1998}, studying the height zeta function with the help of Fourier analysis. In \cite{Batyrev1998b} they explain how this result can be applied to prove Manin's conjecture for our singular surface $S$. Similar methods work for other varieties that are equivariant compactifications of certain algebraic groups (e.g. \cite{Chambert-Loir2002}).

Salberger \cite{Salberger1998} gave a new proof of Manin's conjecture for split toric varieties over the field $\Q$ of rational numbers by a fundamentally different approach using universal torsors. These were first introduced by Colliot-Th\'el\`ene and Sansuc \cite{Colliot1980, Colliot1987} to study the Hasse principle. In the context of Manin's conjecture, the basic idea is to find a parametrization of the rational points on the variety under consideration that makes it feasible to count them by analytic number theory. 

Based on Salberger's ideas, proofs were found for several hard special cases of Manin's conjecture over $\Q$, to which the methods of Batyrev and Tschinkel can not be applied (e.g. \cite{Baier2011, Breteche2002, Breteche2008, Breteche2007c, Breteche2004, Breteche2012, Browning2009, Boudec2011a}). For our surface $S$, independent proofs of Manin's conjecture over $\Q$ were given by de la Bret\`eche \cite{Breteche1998}, Fouvry \cite{Fouvry1998}, Salberger \cite{Salberger1998}, Heath-Brown and Moroz \cite{Heath-Brown1999}, and de la Bret\`eche and Swinnerton-Dyer \cite{Breteche2007}, with the help of such parametrizations. The best error terms have been obtained in \cite{Breteche1998, Breteche2007}. 

In a first attempt to generalize universal torsor techniques to number fields other than $\Q$, Derenthal and Janda \cite{Derenthal2011} modified the approach by Heath-Brown and Moroz \cite{Heath-Brown1999} and successfully applied it to the case of imaginary quadratic number fields of class number $1$.

In this article, we combine the method of Derenthal and Janda with ideas developed by Schanuel \cite{Schanuel1979} and apply it to arbitrary number fields. To the author's best knowledge, this is the first example of universal torsor techniques applied to a special case of Manin's conjecture over general number fields, aside from Schanuel's result for $\P^n$. Hopefully, similar approaches will lead to results for non-toric varieties. 

Before we state the theorem, let us fix some notation: By $\Delta_K$, $h_K$, $R_K$, and $\omega_K$, we denote the discriminant, class number, regulator, and number of roots of unity of $K$. Moreover, $r$ and $s$ denote the number of real and complex places of $K$, and $q := r+s-1$. We write $\O_K$ for the ring of integers of $K$ and $\Norm\ida$ for the absolute norm of the nonzero fractional ideal $\ida$ of $K$.

\begin{maintheorem}\label{main_theorem}
For every number field $K$, we have
\[N(B) = c_K B (\log B)^6 + O(B (\log B)^5),\]
for $B \geq e$. Here, the implicit $O$-constant depends on $K$, and
\[c_K := \frac{9^q}{4\cdot 6!}\left(\frac{2^r (2\pi)^s}{\sqrt{|\Delta_K|}}\right)^9\left(\frac{h_K R_K}{\omega_K}\right)^7 \prod_{\p}\left(1 - \frac{1}{\Norm\p}\right)^7\left(1+\frac{7}{\Norm\p} + \frac{1}{\Norm\p^2}\right)\text,\]
where the product runs over all nonzero prime ideals $\p$ of $\O_K$.
\end{maintheorem}

\subsection{The leading constant}
Let us check the leading constant $c_K$ in Theorem \ref{main_theorem} against the expected one. According to \cite[Section 3.4, Step 4]{Batyrev1998b}, it should have the form 
\[\frac{\gamma_{\mathcal{K}^{-1}}(U) \delta_{\mathcal{K}^{-1}}(U) \tau_{\mathcal{K}^{-1}}(U)}{6!}\text,\]
where $\gamma_{\mathcal{K}^{-1}}(U)$ is the volume of a certain polytope depending only on $U$, $\delta_{\mathcal{K}^{-1}}(U)$ is a cohomological invariant, and $\tau_{\mathcal{K}^{-1}}(U)$ is a generalized version of the Tamagawa number introduced by Peyre \cite{Peyre1995} for smooth Fano varieties.

Derenthal and Janda \cite[Section 3]{Derenthal2011} computed these constants for our $U$ over arbitrary number fields $K$, using a minimal desingularization $\tilde S$ of $S$ constructed by blow-ups of $\P^2$ in six rational points: We have  $\delta_{\mathcal{K}^{-1}}(U) = 1$, and, as already given in \cite[Section 5.3]{Batyrev1998b}, $\gamma_{\mathcal{K}^{-1}}(U) = 1/36$. The Tamagawa number $\tau_{\mathcal{K}^{-1}}(U)$ is an adelic invariant given as a product of local densities with certain convergence factors
\[\tau_{\mathcal{K}^{-1}}(U) = \left(\frac{2^r(2\pi)^s h_K R_K}{\omega_K \sqrt{|\Delta_K|}}\right)^7|\Delta_K|^{-1}\prod_{\nu \mid \infty}\omega_{\mathcal{K}^{-1},\nu}(\tilde{S}(K_\nu))\prod_{\nu \nmid \infty}\lambda_\nu^{-1}\omega_{\mathcal{K}^{-1},\nu}(\tilde{S}(K_\nu))\text.\]
For the Archimedean densities, we have
\[\omega_{\mathcal{K}^{-1}, \nu}(\tilde{S}(K_\nu)) = 36\text{, if $K_\nu = \R$,\quad and \quad}\omega_{\mathcal{K}^{-1}, \nu}(\tilde{S}(K_\nu)) = 36\pi^2\text{, if $K_\nu = \C$.}\]
The non-Archimedean density at the place $\nu$ corresponding to the prime ideal $\p$ of $\O_K$ is given by
\[\lambda_{\nu}^{-1}\omega_{\mathcal{K}^{-1}, \nu}(\tilde{S}(K_\nu)) = \left(1 - \frac{1}{\Norm \p}\right)^{7}\left(1 + \frac{7}{\Norm \p} + \frac{1}{\Norm\p^2}\right)\text.\]
Putting this together, we see that the constant $c_K$ in Theorem \ref{main_theorem} is as expected.

\subsection{More notation}
The ideal class of a nonzero fractional ideal $\ida$ of $K$ is denoted by $[\ida]$. We write $P_K$ for the group of nonzero principal fractional ideals of $K$. We denote the real embeddings by $\sigma_1$, $\ldots$, $\sigma_r : K \to \R$ and the complex embeddings by $\sigma_{r+1}$, $\overline{\sigma}_{r+1}$, $\ldots$, $\sigma_{r+s}$, $\overline{\sigma}_{r+s} : K \to \C$. The component-wise continuation of $\sigma_i$ to $K^n$ is also denoted by $\sigma_i$. If $\nu$ is the place corresponding to $\sigma_i$ then we put $d_i := d_\nu$. When convenient, we write $\alpha^{(i)} := \sigma_i(\alpha)$, for $\alpha \in K$. If $\ida$, $\idb$ are fractional ideals of $K$, we put $(\ida, \idb) := \ida + \idb$. For any point ${\bf x} = (x_0, \ldots, x_n) \in K^{n+1}$, let $\ideal({\bf x}) := (x_0\O_K, \ldots, x_n \O_K)$. Then, for ${\bf x} \in K^4$,
\[H({\bf x}) = \Norm\ideal({\bf x})^{-1}\prod_{i=1}^{r+s} \max\{|x_0^{(i)}|, |x_1^{(i)}|, |x_2^{(i)}|, |x_3^{(i)}|\}^{d_i}\text.\]

We fix, once and for all, a system of fundamental units of $\O_K$, and denote by $\mathcal{F}$ the multiplicative subgroup of $K^\times$ generated by this system. Then $\mathcal{F}$ is a free Abelian group of rank $q$, and the unit group $\O_K^\times$ is the direct product $\O_K^\times = \mu_K \mathcal{F}$, where $\mu_K$ is the group of roots of unity in $K$.

Moreover, we fix, once and for all, a system $\mathcal{C}$ of integral representatives for the ideal classes of $\O_K$, that is a set of $h_K$ nonzero ideals of $\O_K$, one from every ideal class.  

\section{Passing to a universal torsor}
In this section, we find a parametrization of the rational points of bounded height on $U$ by (almost) integral points on an open subset of $\mathbb{A}_K^9$, subject to some height- and coprimality conditions, and up to a certain action of $(\O_K^\times)^7$. This parametrization has the merit that, due to the coprimality conditions, the non-Archimedean parts of the height conditions are trivial.

Over $\Q$ and imaginary quadratic number fields, the action of $(\O_K^\times)^7$ makes no problems, since then $\O_K^\times$ is finite. In general, that is not the case; this is one of the main difficulties which we have to overcome.  

While we will use purely number-theoretic arguments, it should be mentioned that the open subset of $\mathbb{A}^9$ is a universal torsor over $S$, and that our construction is motivated by geometric considerations (see \cite{Derenthal2011}). The choice of indices might seem slightly counter-intuitive at the beginning. It is, however, closely related to  those geometric considerations and will lead to a rather symmetric result.

\subsection{Parametrization}
Let $\Psi_0 : K^3 \to K^4$ be given by
\[\Psi_0(x_{23} , x_{31} , x_{12}) = (x_{12}x_{23}x_{31} , x_{12}x_{31}^2 , x_{23}x_{12}^2 , x_{31}x_{23}^2)\text.\]
We will also consider $\Psi_0$ as a rational map $\P^2 \dasharrow \P^3$. Let $W \subseteq \P^2$ be the open subset
\[W = \{(x_{23} : x_{31} : x_{12}) \in \P^2 \mid x_{12}x_{23}x_{31} \neq 0\}\text.\]
Then $\Psi_0$ induces a bijection between $W(K) \subseteq \P_2(K)$ and $U(K) \subseteq \P_3(K)$ with inverse $(x_0:x_1:x_2:x_3) \mapsto (x_0^2:x_0x_1:x_1x_2)$. Therefore, 
\begin{equation}\label{NBeqVKB}
N(B) = |\{{\bf x} \in W(K) \mid H(\Psi_0({\bf x})) \leq B \}|\text.
\end{equation}

Whenever indices $j$, $k$, $l$ appear in an expression, this expression is understood to hold for all $(j, k, l)\in \{(1, 2, 3), (2, 3, 1), (3, 1, 2)\} =: A$.

\begin{lemma}\label{lemma_ideal_decomposition}
Let $\idb_1$, $\idb_2$, $\idb_3$ be nonzero ideals of $\O_K$, and let $\idc := (\idb_1, \idb_2, \idb_3)$. Then there exist unique nonzero ideals $\ida_1$, $\ida_2$, $\ida_3$, $\ida_{12}$, $\ida_{21}$, $\ida_{23}$, $\ida_{32}$, $\ida_{31}$, $\ida_{13}$ of $\O_K$ such that
\begin{equation}\label{bj_decomposition}
\idb_j = \idc\cdot \ida_{jk}\cdot \ida_{k}^2\cdot \ida_{lk}\cdot \ida_{j}\cdot \ida_{kj}\text,
\end{equation}
and such that the following coprimality conditions hold:

\begin{minipage}{0.3\linewidth} 
\begin{align}
\label{coprim_j_1} (\ida_k, \ida_j) &= \O_K\\
\label{coprim_j_2} (\ida_k, \ida_{kj}) &= \O_K\\
\label{coprim_j_3} (\ida_k, \ida_{jl}) &= \O_K
\end{align}
\vspace{0.2cm}
\end{minipage}
\hspace{0.2cm}
\begin{minipage}{0.3\linewidth} 
\begin{align}
\label{coprim_j_4} (\ida_k, \ida_{lj}) &= \O_K\\
\label{coprim_j_5} (\ida_k, \ida_{kl}) &= \O_K\\
\label{coprim_kj_1} (\ida_{lk}, \ida_{jk}) &= \O_K
\end{align}
\vspace{0.2cm}
\end{minipage}
\hspace{0.2cm}
\begin{minipage}{0.3\linewidth} 
\begin{align}
\label{coprim_kj_2} (\ida_{lk}, \ida_{lj}) &= \O_K\\
\label{coprim_kj_3} (\ida_{lk}, \ida_{jl}) &= \O_K\\
\label{coprim_jk_1} (\ida_{jk}, \ida_{kl}) &= \O_K\text.
\end{align}
\vspace{0.2cm}
\end{minipage}
Conversely, given ideals $\ida_k$, $\ida_{jk}$, $\ida_{lk}$ as in \eqref{coprim_j_1} -- \eqref{coprim_jk_1}, the ideals $\idb_j$ defined by \eqref{bj_decomposition} satisfy $(\idb_1, \idb_2, \idb_3) = \idc$. 
\end{lemma}

\begin{proof}
It is enough to prove the lemma if $\idc = \O_K$, since we can always replace $\idb_j$ by $\idc^{-1}\idb_j$. In this case, we have $(\idb_j, \idb_k^2)(\idb_l, \idb_j) \mid \idb_j$. Let
\begin{equation}\label{def_decomposition_aj}
\ida_{jk} := \frac{\idb_j}{(\idb_j, \idb_k^2)(\idb_l, \idb_j)}\text{, }\ida_k := \left(\frac{\idb_j}{(\idb_j, \idb_k)}, \idb_k\right) \text{, and }\ida_{lk} := \frac{(\idb_j, \idb_k)}{\ida_k}\text.
\end{equation}
Then the $\ida_{jk}$, $\ida_{k}$, $\ida_{lk}$ are nonzero ideals of $\O_K$ and \eqref{bj_decomposition} holds, since $(\idb_j, \idb_k^2) = (\idb_j, \idb_k)\ida_k = \ida_k^2 \ida_{lk}$ and $(\idb_l, \idb_j) = \ida_j\ida_{kj}$. One readily verifies that the left-hand side in conditions \eqref{coprim_j_1} -- \eqref{coprim_j_4}, \eqref{coprim_kj_2}, \eqref{coprim_kj_3} divides $(\idb_1, \idb_2, \idb_3) = \O_K$. Similarly, the left-hand side in \eqref{coprim_j_5}, \eqref{coprim_jk_1} divides $(\idb_j/(\idb_j, \idb_k), \idb_k/(\idb_j, \idb_k)) = \O_K$, and the left-hand side in \eqref{coprim_kj_1} divides $(\idb_k/\ida_k, \idb_j/((\idb_j, \idb_k)\ida_k)) = \O_K$.

Now assume that \eqref{bj_decomposition} holds, with given nonzero ideals $\ida_k$, $\ida_{jk}$, $\ida_{lk}$ satisfying the coprimality conditions \eqref{coprim_j_1} -- \eqref{coprim_jk_1}. These conditions imply that $(\idb_j, \idb_k) = \ida_k\ida_{lk}$, and furthermore $(\idb_j/(\ida_k\ida_{lk}), \idb_k) = \ida_k$. Thus, the $\ida_k$,$\ida_{lk}$ are as in \eqref{def_decomposition_aj}. Clearly, this holds as well for the $\ida_{jk}$, and uniqueness is proved.

The last assertion is again a direct consequence of \eqref{coprim_j_1} -- \eqref{coprim_jk_1}.
\end{proof}
\begin{figure}[ht]\label{coprim_graph}
\[
\xymatrix@C-=20pt@R-=7pt{
1 \ar@{-}[r] & 21 \ar@{-}[rr] && 12 \ar@{-}[r] & 2 \ar@{-}[d]\\
31 \ar@{-}[u] & 13 \ar@{-}[l] & 3 \ar@{-}[l] & 23 \ar@{-}[l] & 32\ar@{-}[l]
}
\]
\caption{The graph $G = (V, E)$.}
\end{figure}
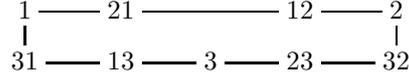
The coprimality conditions \eqref{coprim_j_1} -- \eqref{coprim_jk_1} can be expressed in a more convenient way: Let $G = (V, E)$ be the graph with vertex set $V := \{1, 2, 3, 12, 21, 23, 32, 31, 13\}$ and edge set
\[E := \{\{k, jk\}, \{k, lk\}, \{kl, lk\} \mid (j,k,l) \in A\}\text.\]
Then \eqref{coprim_j_1} -- \eqref{coprim_jk_1} hold if and only if $(\ida_v, \ida_w) = \O_K$ for all pairs $(v, w)$ of nonadjacent vertices of $V$. If we denote the edge set of the complement graph by $E'$, this means that
\begin{equation}\label{coprim_cond_graph}
\text{for any } \{v, w\}\in E'\text{, we have }(\ida_v, \ida_w) = \O_K\text. 
\end{equation}

For every point $(x_{23} : x_{31} : x_{12}) \in W(K)$, the ideal class $[\ideal(x_{23}, x_{31}, x_{12})]$ is well defined, and $[\ideal(x_{23}, x_{31}, x_{12})] = [C]$, for some $C \in \mathcal{C}$. By multiplying with a suitable element of $K^\times$, we can choose a representative ${\bf x} = (x_{23}, x_{31}, {x_{12}}) \in (\O_K\setminus\{0\})^3$ with $\ideal({\bf x}) = C$. This representative is unique up to scalar multiplication by units in $\O_K^\times$.

We apply Lemma \ref{lemma_ideal_decomposition} to the principal ideals $\idb_j := x_{jk}\O_K$ and obtain
\[x_{jk}\O_K = C\cdot \ida_{jk}\cdot \ida_{k}^2\cdot \ida_{lk}\cdot \ida_{j}\cdot \ida_{kj}\text,\]
with unique ideals $\ida_v$ of $\O_K$ satisfying \eqref{coprim_cond_graph}. For all $v \in V \setminus\{12, 23, 31\}$, there is a unique $C_v \in \mathcal{C}$ with $[\ida_v] = [C_v^{-1}]$. Choose $y_v \in K^\times$ with $y_v \O_K = \ida_v C_v$, and define $y_{12}$, $y_{23}$, $y_{31} \in K^\times$ by the equations
\begin{equation}\label{factor_xjk}
x_{jk} = y_{jk}\cdot y_{k}^2\cdot y_{lk}\cdot y_{j}\cdot y_{kj}\text.
\end{equation}
Then
\[y_{jk}\O_K = \ida_{jk} C_{jk}\text{, with }C_{jk} := CC_k^{-2}C_{lk}^{-1}C_j^{-1}C_{kj}^{-1}\text.\]

For ${\bf C} = (C, C_1, C_2, C_3, C_{21}, C_{32}, C_{13}) \in \mathcal{C}^7$, we define $M_{\bf C}$ as the set of all ${\bf y} = (y_v)_{v \in V} \in (K^\times)^9$ such that
\begin{equation}\label{def_MC}
y_v \in C_v \text{ for all }v \in V\text{, and the ideals }\ida_v := y_v C_v^{-1}\text{ satisfy \eqref{coprim_cond_graph}.}
\end{equation}
By what we have shown above, relations \eqref{factor_xjk} define a surjective mapping
\[\phi : \bigcup_{{\bf C} \in \mathcal{C}^7}M_{\bf C} \to W(K)\text.\]
If ${\bf y} \in M_{\bf C}$ and $\phi({\bf y}) = (x_{23} : x_{31} : x_{12})$ with $x_{jk}$ as in \eqref{factor_xjk} then 
\[x_{jk}\O_K = C\cdot \ida_{jk}\cdot\ida_k^2\cdot \ida_{lk}\cdot\ida_j\cdot\ida_{kj}\text.\]
By Lemma \ref{lemma_ideal_decomposition}, we have $\ideal(x_{23}, x_{31}, x_{12}) = C$, and the $\ida_v$ (and thus as well the $C_v$) are uniquely determined by the $x_{jk}\O_K$. In particular, the sets $M_{\bf C}$, ${\bf C} \in \mathcal{C}^7$, are pairwise disjoint. Moreover, $(x_{23}, x_{31}, x_{12})$ and the $y_v$, $v \in V$, are determined by $\phi({\bf y})$ up to multiplication by units. Therefore, $\phi({\bf y}) = \phi({\bf z})$ if and only if there are units $\zeta$, $\zeta_v \in \O_K^\times$ with
\[z_v = \zeta_v y_v \text{ for all }v \in V\text{ and }\zeta_{jk}\zeta_k^2\zeta_{lk}\zeta_j\zeta_{kj} = \zeta\text{ for all }(j, k, l) \in A\text.\]   
By eliminating the $\zeta_{jk}$, we see that $\phi({\bf y}) = \phi({\bf z})$ if and only if ${\bf y}$ and ${\bf z}$ are in the same orbit of the action $\odot$ of $(\O_K^\times)^7$ on $(K^\times)^9$ given by
\begin{equation}\label{action_unit_group}
(\zeta, \zeta_1, \zeta_2, \zeta_3, \zeta_{21}, \zeta_{32}, \zeta_{13}) \odot (y_v)_v := (z_v)_v\text,
\end{equation}
where $z_v := \zeta_v y_v$ for all $v \in V \setminus\{12, 23, 31\}$ and $z_{jk} := \zeta \zeta_k^{-2}\zeta_{lk}^{-1}\zeta_j^{-1}\zeta_{kj}^{-1} y_{jk}$. 

In what follows, it will be more convenient to work with the free Abelian subgroup $\mathcal{F}$ of $\O_K^\times$ generated by our fixed system of fundamental units. Clearly, $(\O_K^\times)^7$ is the direct product $(\O_K^\times)^7 = \mu_K^7 \cdot \mathcal{F}^7$. Since the action of $(\O_K^\times)^7$ on $(K^\times)^9$ is free, every orbit of $(K^\times)^9$ under the action of $(\O_K^\times)^7$ is the union of $|\mu_K^7| =\omega_K^7$ orbits under the action of $\mathcal{F}^7$.

Let $\mathcal{R}$ be a system of representatives for the orbits of $(K^\times)^9$ under the action of $\mathcal{F}^7$. Then $\phi$ induces an $\omega_K^7$-to-$1$ map
\[\phi : \bigcup_{{\bf C} \in \mathcal{C}^7}(M_{\bf C} \cap \mathcal{R}) \to W(K)\text.\]

The benefits of our construction become apparent in the height condition. With ${\bf x} = (x_{23}, x_{31}, x_{12})$ as in \eqref{factor_xjk}, we have 
$\psi_0({\bf x}) = y_1^2 y_2^2 y_3^2 y_{21} y_{32} y_{13}\cdot \psi({\bf y})$, where
\[\psi({\bf y}) = (\psi({\bf y})_0,\psi({\bf y})_1,\psi({\bf y})_2,\psi({\bf y})_3)\]
with
\[\psi({\bf y})_0 := \prod_{v \in V}y_v \text{ and }\psi({\bf y})_j := y_j^3 y_{jk}y_{jl}y_{kj}^2y_{lj}^2\text{, for }1 \leq j \leq 3\text.\]
Therefore, 
\[H(\psi_0({\bf x})) = H(\psi({\bf y})) = \Norm \ideal(\psi({\bf y}))^{-1}\prod_{i=1}^{r+s} \max_{0\leq j\leq 3}\{|\psi({\bf y})_j^{(i)}|\}^{d_i}\text.\]
A straightforward computation using $y_v = \ida_v C_v$ and \eqref{coprim_cond_graph} shows that 
\[\ideal(\psi({\bf y})) = C^3C_1^{-2}C_2^{-2}C_3^{-2}C_{21}^{-1}C_{32}^{-1}C_{13}^{-1}\text.\]
By our construction, $\psi({\bf y})$ satisfies the equation $\psi({\bf y})_0^3 = \psi({\bf y})_1\psi({\bf y})_2\psi({\bf y})_3$. Since this holds as well for all conjugates, the maximum is always one of $|\psi({\bf y})_1^{(i)}|$, $|\psi({\bf y})_2^{(i)}|$, $|\psi({\bf y})_3^{(i)}|$. We define 
\begin{equation}\label{def_MRB}\mathcal{R}(B) := \{{\bf y} \in \mathcal{R} \mid \prod_{i=1}^{r+s} \max_{1 \leq j \leq 3}\{|\sigma_i(y_j^3 y_{jk}y_{jl}y_{kj}^2y_{lj}^2)|\}^{d_i} \leq B\}\text.
\end{equation}
The results of this section can be summarized as follows.
\begin{proposition}\label{count_torsor}
Let $M_{\bf C}$ be as in \eqref{def_MC}, let $\mathcal{R}$ be any system of representatives for the orbits of $(K^\times)^9$ under the action $\odot$ of $\mathcal{F}^7$ given by \eqref{action_unit_group}, and let $\mathcal{R}(B)$ be as in \eqref{def_MRB}. Then $M_{\bf C} \cap \mathcal{R}(B)$ is finite for all $B>0$, ${\bf C} \in \mathcal{C}^7$, and
\[N(B) = \frac{1}{\omega_K^7}\sum_{\substack{{\bf C} \in \mathcal{C}^7}} |M_{\bf C} \cap \mathcal{R}(u_{\bf C}B)|\text,\]
where
\[u_{\bf C} := \Norm(C^3 C_1^{-2} C_2^{-2} C_3^{-2} C_{21}^{-1} C_{32}^{-1} C_{13}^{-1})\text.\]
\end{proposition}

\subsection{A system of representatives for the orbits}\label{subsection_fundamental_domain}
We construct a system $\mathcal{R}$ of representatives for the orbits of $(K^\times)^9$ under the action $\odot$ of $\mathcal{F}^7$ given by \eqref{action_unit_group}. 

\begin{lemma}\label{fund_syst_equations}
Let $\alpha_1$, $\alpha_2$, $\alpha_3 \in \mathcal{F}$ and consider the system of equations
\begin{equation}
\label{fund_sys_eq}\zeta \zeta_k^{-2} \zeta_{j}^{-1} = \alpha_j\text{, for }(j, k) \in \{(1,2), (2,3), (3,1)\}\text,
\end{equation}
with variables $\zeta$, $\zeta_j \in \mathcal{F}$. 
\begin{enumerate}[(i)]
 \item If $\alpha_1\alpha_2\alpha_3$ is not a cube in $\mathcal{F}$ then this system has no solutions.
 \item If $\alpha_1\alpha_2\alpha_3 = \xi^3$ with $\xi \in \mathcal{F}$ then the solutions are given by
 
 \begin{minipage}{0.3\linewidth}
 \begin{align*}
  \zeta_1 &= \delta\vphantom{\xi \alpha_2^{-1}}\\
  \zeta_2 &= \delta \xi^{-1}\alpha_3
 \end{align*}
 \end{minipage}
 \hspace{0.1cm}
 \begin{minipage}{0.3\linewidth}
 \begin{align*}
  \zeta_3 &= \delta \xi \alpha_2^{-1}\\
  \zeta &= \delta^3 \xi \alpha_2^{-1}\alpha_3\text, 
 \end{align*}
 \end{minipage}\\
\\
for all $\delta \in \mathcal{F}$.
\end{enumerate}
\end{lemma}

\begin{proof}
Equations \eqref{fund_sys_eq} imply that
\begin{equation}\label{fund_sys_eq_1}
\zeta^{3} \zeta_{j}^{-9} = \alpha_j \alpha_k^{-2} \alpha_l^4 = \alpha_1\alpha_2\alpha_3 \alpha_k^{-3} \alpha_l^3 \text,
\end{equation}
which proves \emph{(i)}.

Now assume that $\alpha_1\alpha_2\alpha_3 = \xi^3$ for some $\xi \in \mathcal{F}$. Then $\xi$ is unique since $\mathcal{F}$ is free Abelian. Direct computations verify that the values given in \emph{(ii)} are solutions.

Given any solution $(\zeta, \zeta_1, \zeta_2, \zeta_3)$ of \eqref{fund_sys_eq}, let $\delta := \zeta_1$. Then \eqref{fund_sys_eq_1} with $j=1$ shows that $\zeta$ has the desired form. Similar computations using \eqref{fund_sys_eq_1} with $j = 2$ and $j = 3$ prove that $\zeta_2$ and $\zeta_3$ are as desired.
\end{proof}

Let $H$ be the subgroup of $(K^\times)^6$ of all $\underline{\alpha} = (\alpha_{12},\alpha_{21},\alpha_{23},\alpha_{32},\alpha_{31},\alpha_{13}) \in \mathcal{F}^6$ for which $\alpha_{12}\alpha_{21}^2\alpha_{23}\alpha_{32}^2\alpha_{31}\alpha_{13}^2$ is a cube in $\mathcal{F}$.

\begin{lemma}\label{lemma_fundamental_domain}
Let $\mathcal{R}_1 \subseteq (K^\times)^3$ be a system of representatives for the orbits of $(K^\times)^3$ under the action of $\mathcal{F}$ by scalar multiplication, and let $\mathcal{R}_2 \subseteq (K^\times)^6$ be a system of representatives for $(K^\times)^6/H$. Then $\mathcal{R} := \mathcal{R}_1 \times \mathcal{R}_2$ 
is a system of representatives for the orbits of $(K^\times)^9$ under the action $\odot$ of $\mathcal{F}^7$.
\end{lemma}

\begin{proof}
Let ${\bf y} = (y_v)_{v \in V} \in (K^\times)^9$. Then there is a unique $\underline{\alpha} \in H$ such that 
\[(\alpha_{12} y_{12}, \alpha_{21} y_{21}, \alpha_{23} y_{23}, \alpha_{32} y_{32}, \alpha_{31} y_{31}, \alpha_{13} y_{13}) \in \mathcal{R}_2\text.\]
The elements $\underline{\zeta} = (\zeta, \zeta_1, \zeta_2, \zeta_3, \zeta_{21}, \zeta_{32}, \zeta_{13}) \in \mathcal{F}^7$ with 
\[\underline{\zeta} \odot {\bf y} \in (K^\times)^3 \times \mathcal{R}_2\]
are those satisfying
\begin{equation}\label{fund_sys_eq_2}
\zeta_{kj} = \alpha_{kj} \text{ and }\zeta\zeta_k^{-2}\zeta_{lk}^{-1}\zeta_j^{-1}\zeta_{kj}^{-1} = \alpha_{jk}\text.
\end{equation}
With $\alpha_j := \alpha_{jk}\alpha_{kj}\alpha_{lk}$, this simplifies to \eqref{fund_sys_eq}. Now 
\[\alpha_1\alpha_2\alpha_3 = \alpha_{12}\alpha_{21}^2\alpha_{23}\alpha_{32}^2\alpha_{31}\alpha_{13}^2\]
is a cube in $\mathcal{F}$, so $\zeta$, $\zeta_1$, $\zeta_2$, $\zeta_3$ are of the form given in Lemma \ref{fund_syst_equations}, \emph{(ii)}, for $\delta \in \mathcal{F}$. There is exactly one $\delta \in \mathcal{F}$ such that the corresponding $\zeta_1$, $\zeta_2$, $\zeta_3$ satisfy $(\zeta_1 y_1, \zeta_2 y_2, \zeta_3 y_3) \in \mathcal{R}_1$. Hence, there is exactly one $\underline{\zeta} \in \mathcal{F}^7$ with $\underline{\zeta} \odot {\bf y} \in \mathcal{R}$.
\end{proof}

\begin{lemma}\label{lemma_R}
Let $R \subseteq K^\times$ be a system of representatives for $K^\times / \mathcal{F}$, and let $R_\mathcal{F} \subseteq \mathcal{F}$ be a system of representatives for $\mathcal{F}/\{\xi^3 \mid \xi \in\mathcal{F}\}$. Then
\[\mathcal{R}_2 := \bigcup_{\rho \in R_\mathcal{F}} (\rho R \times R \times R \times R \times R \times R)\]
is a system of representatives for $(K^\times)^6/H$.
\end{lemma}

\begin{proof}
Clearly, $\bigcup_{\rho \in \mathcal{R}_\mathcal{F}}\rho R$ is a system of representatives for $K^\times/\{\xi^3 \mid \xi \in \mathcal{F}\}$. Let ${\bf y} \in (K^\times)^6$. For all $v \in \{21, 23, 32, 31, 13\}$, there is exactly one $\alpha_v \in \mathcal{F}$ with $\alpha_v y_v \in R$. Moreover, there is exactly one $\xi \in \mathcal{F}$ such that $y_{12}(\alpha_{21}^2\alpha_{23}\alpha_{32}^2\alpha_{31}\alpha_{13}^2)^{-1}\xi^3 \in \bigcup_{\rho \in \mathcal{R}_\mathcal{F}}\rho R$. Hence, there is exactly one $\alpha_{12} := (\alpha_{21}^2\alpha_{23}\alpha_{32}^2\alpha_{31}\alpha_{13}^2)^{-1}\xi^3 \in \mathcal{F}$, such that $\underline{\alpha} = (\alpha_{12}, \alpha_{21}, \alpha_{23}, \alpha_{32}, \alpha_{31}, \alpha_{13}) \in H$ and $\underline{\alpha}{\bf y} \in \mathcal{R}_2$. 
\end{proof}

We choose the system $\mathcal{R} = \mathcal{R}_1 \times \mathcal{R}_2$ as in Lemma \ref{lemma_fundamental_domain}, where $\mathcal{R}_1$ is any system of representatives for the diagonal action of $\mathcal{F}$ on $(K^\times)^3$, and $\mathcal{R}_2$ is as in Lemma \ref{lemma_R}. 

\section{Proof of Theorem 1}
This section is a generalization of \cite[Section 5]{Derenthal2011}. We reduce Theorem \ref{main_theorem} to a central lemma (Lemma \ref{main_lemma}), whose proof will take up the rest of the article. We assume that $K$ is of degree $d \geq 2$. Over $\Q$, one would need to replace Lemma \ref{lemma_y1y2y3} by a slightly more intricate argument to make the sum over the error terms converge, for which we refer to \cite{Heath-Brown1999}.   

\subsection{M\"obius inversions}
Let ${\bf C} = (C, C_1, C_2, C_3, C_{21}, C_{32}, C_{13}) \in \mathcal{C}^7$ be fixed. We investigate the quantity $|M_{\bf C} \cap \mathcal{R}(u_{\bf C}B)|$ from Proposition \ref{count_torsor}. We can write
\[|M_{\bf C} \cap \mathcal{R}(u_{\bf C}B)| = \sum_{\substack{{\bf y} \in \mathcal{R}(u_{\bf C}B)\\\eqref{def_MC}\text{ holds}}}1\text.\]
M\"obius inversion for all the coprimality conditions in \eqref{coprim_cond_graph} yields
\begin{equation}\label{moebius_1}
|M_{\bf C} \cap \mathcal{R}(u_{\bf C}B)| = \sum_{\substack{(\mathfrak{d}_e)_{e \in E'}\\\{0\} \neq \mathfrak{d}_e\unlhd \O_K}}\left(\prod_{e \in E'}\mu(\mathfrak{d}_e)\right)\sum_{\substack{{\bf y} \in \mathcal{R}(u_{\bf C}B)\\\forall e = \{v, w\} \in E':\ y_v \in \mathfrak{d}_eC_v\text{, }y_w \in \mathfrak{d}_eC_w}}1\text,\end{equation}
where each $\mathfrak{d}_e$ runs over all nonzero ideals of $\O_K$ and $\mu$ is the M\"obius function for nonzero ideals of $\O_K$. Lemma \ref{main_lemma} will imply that the last sum is always finite and different from zero for at most finitely many $(\mathfrak{d}_e)_{e \in E'}$. With $\ida_v := \bigcap_{v \in e \in E'}\mathfrak{d}_eC_v$, we obtain
\begin{equation}\label{moebius}
\sum_{\substack{{\bf y} \in \mathcal{R}(u_{\bf C}B)\\\forall e = \{v, w\} \in E':\ y_v \in \mathfrak{d}_eC_v\text{, }y_w \in \mathfrak{d}_eC_w}}1 = \sum_{\substack{{\bf y} \in \mathcal{R}(u_{\bf C}B)\\\forall v:\ y_v \in \ida_v}}1\text.\end{equation}
We estimate this sum by the following lemma. Its proof is central to this article and will be given in Section \ref{sec_proof_lemma}.
\begin{lemma}\label{main_lemma}
For every $v \in V$, let $\ida_v$ be a fractional ideal of $K$ with $\Norm\ida_v \geq c$, for some constant $c > 0$ depending only on $K$. With $\mathcal{R}(B)$ as in \eqref{def_MRB}, we have
\begin{align*}\sum_{\substack{{\bf y} \in \mathcal{R}(B)\\\forall v:\ y_v \in \ida_v}}1 &= \frac{9^q}{4\cdot 6!} \left(\frac{2^r(2\pi)^s}{\sqrt{|\Delta_K|}}\right)^9\frac{R_K^7}{\prod_{v\in V}\Norm\ida_v} B(\log B)^6\\ &+ O\left(\frac{\max_j\{\Norm \ida_j\}^{1/d}}{\prod_j\Norm\ida_j \prod_{i\neq j}\Norm\ida_{ij}^{1-2/(3d)}} B(\log B)^5\right)\text,\end{align*}
for $B \geq e$. The implicit $O$-constant depends on $K$.
\end{lemma}

\noindent For any $(\mathfrak{d}_e)_{e \in E'}$ and $v \in V$, we define $r_v := \Norm(\cap_{v \in e \in E'}\mathfrak{d}_e)$,
\begin{equation}\label{R1_R2}R_1 := \prod_{v \in V}r_v\text{,}\quad\text{ and }\quad R_2 := \max_j\{r_j\}^{-1/d}\prod_{j}r_j\prod_{i \neq j}r_{ij}^{1-2/(3d)}\text.\end{equation}
We notice that $\Norm\ida_v = \Norm(\cap_{v \in e \in E'}\mathfrak{d}_eC_v) = \Norm(C_v)r_v$. Recall that we defined $C_{jk} := CC_k^{-2}C_{lk}^{-1}C_j^{-1}C_{kj}^{-1}$, for $jk \in \{12,23,31\}$, so
\[\prod_{v \in V}\Norm C_v = \Norm(C^3 C_1^{-2} C_2^{-2} C_3^{-2} C_{21}^{-1} C_{32}^{-1} C_{13}^{-1}) = u_{\bf C}\text.\]
Since the $C$, $C_j$, $C_{kj}$ are members of the fixed finite set $\mathcal{C}$, their absolute norms are bounded from below and above by positive constants depending only on $K$. With this and Lemma \ref{main_lemma}, we obtain
\[\sum_{\substack{{\bf y} \in \mathcal{R}(u_{\bf C}B)\\y_v \in \ida_v}}1 = \frac{9^q}{4\cdot 6!} \left(\frac{2^r(2\pi)^s}{\sqrt{|\Delta_K|}}\right)^9 R_K^7\frac{B}{R_1}(\log B)^6 + O\left(\frac{B}{R_2}(\log B)^5\right)\text,\]
whenever $B \geq e/u_{\bf C}$. Otherwise, the error term dominates the main term. Let
\begin{equation}\label{rho_omega}
\omega := \sum_{\substack{(\mathfrak{d}_e)_{e \in E'}\\\{0\} \neq \mathfrak{d_e}\unlhd \O_K}}\prod_{e \in E'}\mu(\mathfrak{d}_e)R_1^{-1}\text{, }\quad \rho := \sum_{\substack{(\mathfrak{d}_e)_{e \in E'}\\\{0\} \neq \mathfrak{d_e}\unlhd \O_K}}\prod_{e \in E'}|\mu(\mathfrak{d}_e)|R_2^{-1}\text.
\end{equation}
We will see in Lemma \ref{rho_omega_converge} that these sums converge under our assumption that $d \geq 2$. Since the sum defining $\rho$ converges, \eqref{moebius_1} and \eqref{moebius} yield
\begin{align*}
|M_{\bf C} \cap \mathcal{R}(u_{\bf C}B)| &= \frac{9^q}{4\cdot 6!} \left(\frac{2^r(2\pi)^s}{\sqrt{|\Delta_K|}}\right)^9 R_K^7 \omega B(\log B)^6\\ &+ O(B(\log B)^5)\text.
\end{align*}

\subsection{Computation of the constant}

We notice that the above expression for $|M_{\bf C} \cap \mathcal{R}(u_{\bf C}B)|$ does not depend on ${\bf C} \in \mathcal{C}^7$. Therefore, Proposition \ref{count_torsor} implies
\[N(B) = \frac{9^q}{4\cdot 6!} \left(\frac{2^r(2\pi)^s}{\sqrt{|\Delta_K|}}\right)^9 \left(\frac{h_K R_K}{\omega_K}\right)^7 \omega B(\log B)^6 + O(B(\log B)^5)\text.\]
Theorem \ref{main_theorem} is an immediate consequence of the following lemma.
\begin{lemma}\label{rho_omega_converge}
Let $\omega$, $\rho$ be as in \eqref{rho_omega}, with $R_1$, $R_2$ as in \eqref{R1_R2}. If $d \geq 2$ then both sums converge, and 
\begin{equation}\label{euler_product_omega}\omega = \prod_{\p}\left(1 - \frac{1}{\Norm\p}\right)^7\left(1+\frac{7}{\Norm\p} + \frac{1}{\Norm\p^2}\right)\text,\end{equation}
where the product runs over all nonzero prime ideals $\mathfrak{p}$ of $\O_K$.
\end{lemma}

\begin{proof}
The proof is a straightforward generalization of the one in \cite[Section 5]{Derenthal2011}. An obvious modification of the argument given there shows that the Euler factor of $\rho$ corresponding to a prime ideal $\mathfrak{p}$ of $\O_K$ is $1 + O(\Norm\p^{-{(6d-5)/(3d)}})$, so the sum defining $\rho$ is convergent whenever $d \geq 2$. Since $\omega \leq \rho$, the sum defining $\omega$ converges as well. 

Let $A(x)$ be the polynomial defined in \cite[Section 5]{Derenthal2011}, and $A_\p$ the Euler factor of $\omega$ corresponding to $\p$. Then we have $A_\p = A(\Norm\p^{-1})$, and \eqref{euler_product_omega} follows from the investigation of $A(x)$ in \cite[Section 5]{Derenthal2011}.
\end{proof}

This completes our proof of Theorem \ref{main_theorem}, up to proving Lemma \ref{main_lemma}. 

\section{Auxiliary results}
Let $n$, $M$ be positive integers and $L > 0$. By $\Lip(n, M, L)$, we denote the set of all subsets $\mathcal{B}$ of $\R^n$ for which there exist $M$ maps $\Phi : [0, 1]^{n-1} \to \R^n$ satisfying a Lipschitz condition
\[\left|\Phi(v) - \Phi(w)\right| \leq L\left|v - w\right|\text,\]
such that $\mathcal{B}$ is covered by the union of the images of the maps $\Phi$. Here, $|\cdot |$ is the usual Euclidean norm. (The subsets in $\Lip(1, M, L)$ are just those with at most $M$ elements.) We will use the following lemma to bound the error terms when estimating a sum by an integral. Part \emph{(i)} generalizes an argument used in \cite[Chapter VI, Theorem 2]{Lang1994}.

\begin{lemma}\label{translate_of_fixed_set}
 Let $D$, $\mathcal{B} \subseteq \R^n$ be bounded subsets with $\mathcal{B} \in \Lip(n, M, L)$.
\begin{enumerate}[(i)]
 \item Let $\Lambda \subseteq \R^n$ be a lattice. Then
\[|\{\lambda \in \Lambda \mid (\lambda + D) \cap \mathcal{B} \neq \emptyset\}| \ll_{\Lambda, D} M (L+1)^{n-1}\text.\]
\item If $D$, $\mathcal{B}$ are compact then $\{{\bf x}\in \R^n \mid (x + D) \cap \mathcal{B} \neq \emptyset\}$ is measurable and
\[\Vol \{{\bf x}\in \R^n \mid (x + D) \cap \mathcal{B} \neq \emptyset\} \ll_{D} M (L+1)^{n-1}\text.\]
\end{enumerate}
\end{lemma}

\begin{proof}
For $x \in \R^n$, we have $(x + D) \cap \mathcal{B} \neq \emptyset$ if and only if $x \in \mathcal{B} - D$. If $\mathcal{B}$ and $D$ are compact, the set $\mathcal{B} - D$ is compact as well. This proves measurability of the set in \emph{(ii)}.

Let $\Phi : [0,1]^{n-1} \to \R^n$ be one of the $M$ maps with Lipschitz constant $L$ whose images cover $\mathcal{B}$. We split up $[0,1]^{n-1}$ into $L_1^{n-1}$ subcubes of side length $1/L_1$, where $L_1 := \lfloor L \rfloor + 1$. Let $C$ be one of those subcubes. Then $\Phi(C)$ has diameter at most $\sqrt{n-1} L/L_1 \leq \sqrt{n-1}$, so it is contained in a closed ball $B_{\bf z}(2\sqrt{n-1})$ of radius $2\sqrt{n-1}$ centered at some point ${\bf z} \in \R^n$.

Since $D$ is bounded, it is contained in a closed zero-centered ball $B_{\bf 0}(R_D)$ of some radius $R_D$. Every point ${\bf x} \in \R^n$ with $({\bf x} + D) \cap \Phi(C) \neq \emptyset$ satisfies ${\bf x} \in B_{\bf z}(2\sqrt{n-1}) - B_{\bf 0}(R_D) = B_{\bf z}(2\sqrt{n-1} + R_D)$. 

The number of lattice points in such a ball is finite and can be bounded independently from ${\bf z}$. Therefore,
\begin{equation}\label{translate_of_fixed_set_1}|\{\lambda \in \Lambda \mid (\lambda + D) \cap \Phi(C) \neq \emptyset\}| \ll_{\Lambda, D} 1\text.\end{equation}
Moreover
\begin{equation}\label{translate_of_fixed_set_2}\Vol\{{\bf x}\in \R^n \mid (x + D) \cap \Phi(C) \neq \emptyset\} \leq \Vol B_{\bf z}(2\sqrt{n-1} + R_D) \ll_D 1\text.\end{equation}
Summing \eqref{translate_of_fixed_set_1} and \eqref{translate_of_fixed_set_2} over all $C$ and $\Phi$ yields \emph{(i)} and \emph{(ii)}.
\end{proof}

\subsection{Counting lattice points}
We will need to count lattice points in certain bounded subsets of $\R^n$ for lattices $\Lambda \subseteq \R^n$ of the form
\[\Lambda = \Lambda_1 \times \cdots \times \Lambda_r\text,\]
where each $\Lambda_i$ is a lattice in $\R^{n_i}$ and $n_1 + \cdots + n_r = n$. Then we have $\det(\Lambda) = \det(\Lambda_1) \cdots \det(\Lambda_r)$, and the successive minima (with respect to the unit ball) of $\Lambda$ are just the successive minima of $\Lambda_1$, $\ldots$, $\Lambda_r$. Several authors (e.g. \cite{Christensen2008, Masser2006}) provide counting results where the first successive minimum is reflected in the error term, by making an argument from \cite[Chapter VI, Theorem 2]{Lang1994} explicit. For our application, we need the error term to reflect information about all the lattices $\Lambda_i$, which is accomplished with the help of a theorem by Widmer.
\begin{theorem}\cite[Theorem 5.4]{Widmer2010}\label{widmer}
Let $\Lambda$ be a lattice in $\R^n$ with successive minima (with respect to the unit ball) $\lambda_1$, $\ldots$, $\lambda_n$. Let $\mathcal{B}$ be a bounded set in $\R^n$ with boundary $\partial \mathcal{B} \in \Lip(n, M, L)$. Then $\mathcal{B}$ is measurable, and moreover
\[\left|\left|\mathcal{B} \cap \Lambda\right| - \frac{\Vol \mathcal{B}}{\det \Lambda} \right| \leq c_0(n) M \max_{0\leq k<n}\frac{L^k}{\lambda_1\cdots \lambda_k}\text.\]
For $k=0$, the expression in the maximum is to be understood as $1$. Furthermore, one can choose $c_0(n) = n^{3n^2/2}$.
\end{theorem}

Let $\lambda_{i1} \leq \cdots \leq \lambda_{in_i}$ be the successive minima of $\Lambda_i$, and assume that the $\Lambda_i$ are ordered in such a way that $\lambda_{11} \leq \lambda_{21} \leq \cdots \leq \lambda_{r1}$ holds.

\begin{corollary}\label{widmer_corollary}
Let $\Lambda$ and $\Lambda_i$ be as above, and let $\mathcal{B}\subseteq \R^n$ be a bounded set with boundary $\partial \mathcal{B} \in \Lip(n, M, L)$. Then $\mathcal{B}$ is measurable and
\[\left|\left|\mathcal{B} \cap \Lambda\right| - \frac{\Vol \mathcal{B}}{\det \Lambda} \right| \leq c_0(n) M \prod_{i=1}^{r-1}\left(\frac{L}{\lambda_{i1}}+1\right)^{n_i}\left(\frac{L}{\lambda_{r1}}+1\right)^{n_r - 1}\text.\]
\end{corollary}

\begin{proof}
 We use Theorem \ref{widmer}. Let $\lambda_1 \leq \cdots \leq \lambda_n$ be the successive minima of $\Lambda$, that is, the $\lambda_{ij}$ in correct order. Clearly, 
 \begin{align*}
  \max_{0\leq k<n}\frac{L^k}{\lambda_1\cdots \lambda_k} &\leq \prod_{j=1}^{n-1}\left(\frac{L}{\lambda_j} + 1\right) \leq\prod_{i=1}^{r}\left(\frac{L}{\lambda_{i1}} + 1\right)^{n_i}\Big/\left(\frac{L}{\lambda_{i_01}} + 1\right)\text,
 \end{align*}
 where $i_0$ is chosen such that $\lambda_{i_0n_{i_0}} = \lambda_n$. The last expression is at most
 \[\prod_{i=1}^{r-1}\left(\frac{L}{\lambda_{i1}}+1\right)^{n_i}\left(\frac{L}{\lambda_{r1}}+1\right)^{n_r - 1}\text.\]
\end{proof}

\begin{lemma}\label{lattice_pt_homog_exp}
 Let $\Lambda$ and $\Lambda_i$ be as above, and let $\mathcal{B} \subseteq \R^n$ be contained in a zero-centered ball of radius $R$. Assume, moreover, that $\partial \mathcal{B} \in \Lip(n, M, L)$, and that the following property holds for all ${\bf x} \in \mathcal{B}$:
 \begin{equation}\label{nonzero_coordinates}
 \text{If we write ${\bf x} = ({\bf x}_1, \ldots, {\bf x}_r)$ with ${\bf x}_i \in \R^{n_i}$ then ${\bf x}_i \neq {\bf 0}$ for all $i$.}\end{equation}
 Then $\mathcal{B}$ is measurable and, for all $T \geq 0$, we have
 \[\left|\left|T\mathcal{B} \cap \Lambda\right| - \frac{T^n\Vol \mathcal{B}}{\det \Lambda} \right| \ll_{n,M,R,L} \prod_{i=1}^{r-1}\left(\frac{T}{\lambda_{i1}}\right)^{n_i}\left(\frac{T}{\lambda_{r1}}\right)^{n_r-1}\text.\]
\end{lemma}

\begin{proof}
By Theorem \ref{widmer}, $\mathcal{B}$ is measurable. We start with the case where $TR < \lambda_{r1}$. Suppose that ${\bf a} = ({\bf a}_1, \ldots, {\bf a}_r) \in T\mathcal{B} \cap \Lambda$. Then  ${\bf a}_r \neq {\bf 0}$ by \eqref{nonzero_coordinates}. Therefore, $|{\bf a}| \geq |{\bf a}_r| \geq \lambda_{r1} > TR$, so ${\bf a} \notin T\mathcal{B}$, a contradiction. Hence, $|T\mathcal{B} \cap \Lambda| = 0$. Denote by $V_1$ the volume of a ball of radius $1$ in $\R^n$. Then $\Vol \mathcal{B} \leq R^n V_1$. We denote the successive minima of $\Lambda$ again by $\lambda_1$, $\ldots$, $\lambda_n$. By Minkowski's second theorem we have
\[\frac{T^n \Vol \mathcal{B}}{\det \Lambda} \leq \frac{V_1 2^n (RT)^n}{\lambda_1\cdots \lambda_n V_1} \leq 2^n R^{n-1}\prod_{i=1}^{r-1}\left(\frac{T}{\lambda_{i1}}\right)^{n_i}\left(\frac{T}{\lambda_{r1}}\right)^{n_r-1}\text.\]
Now assume $TR \geq \lambda_{r1}$. Clearly, $\Vol(T \mathcal{B}) = T^n\Vol \mathcal{B}$ and $\partial(T \mathcal{B}) \in \Lip(n, M, T L)$. To finish the proof, we use Corollary \ref{widmer_corollary} and observe that
\begin{align*}\prod_{i=1}^{r-1}\left(\frac{TL}{\lambda_{i1}}+1\right)^{n_i}\left(\frac{TL}{\lambda_{r1}}+1\right)^{n_r - 1} &\leq \prod_{i=1}^{r-1}\left(\frac{T(L + R)}{\lambda_{i1}}\right)^{n_i}\left(\frac{T(L + R)}{\lambda_{r1}}\right)^{n_r - 1}\\
&=(L + R)^{n-1} \prod_{i=1}^{r-1}\left(\frac{T}{\lambda_{i1}}\right)^{n_i}\left(\frac{T}{\lambda_{r1}}\right)^{n_r-1}\text.
\end{align*}
\end{proof}

\subsection{The basic sets}\label{subsection_basic_set}
Here, we describe the sets $\mathcal{B}$ to which Lemma \ref{lattice_pt_homog_exp} will be applied. These sets have been introduced by Schanuel \cite{Schanuel1979} and in a more general context by Masser and Vaaler \cite{Masser2006}. Our notation is similar to the one in \cite{Masser2006}. When talking about lattices, volumes, etc., we identify $\C$ with $\R^2$.

Let $\Sigma$ be the hyperplane in $\R^{r+s}$ where $x_1 + \cdots + x_{r+s} = 0$. It is well known that the map $l : K^\times \to \R^{r+s}$ defined by $l(\alpha) = (d_1 \log |\alpha^{(1)}|, \ldots, d_{r+s}\log|\alpha^{(r+s)}|)$ induces a group homomorphism of $\O_K^\times$ onto a lattice in $\Sigma$, with kernel $\mu_K$. In particular, $l$ induces a group isomorphism from $\mathcal{F}$ to $l(\O_K^\times)$. Let $F$ be a fundamental parallelotope for this lattice, and let $\delta := (d_1, \ldots, d_{r+s}) \in \R^{r+s}$. We define the vector sums
\[F(\infty) := F + \R \delta\ \text{, and }\ F(T) := F + (-\infty, \log T]\delta\ \text{, for $T > 0$.}\]
Then $F(\infty)$ is a system of representatives for the orbits of the additive action of $l(\mathcal{F}) = l(\O_K^\times)$ on $\R^{r+s}$. Let $S_F^n(T)$ be the set of all 
\[(z_{1,1}, \ldots, z_{1, n}, \ldots, z_{r+s, 1}, \ldots, z_{r+s,n}) \in (\R^n \setminus \{{\bf 0}\})^r \times (\C^n \setminus \{{\bf 0}\})^s\]
such that 
\[(d_i \log \max_{1 \leq j \leq n}\{|z_{i,j}|\})_{i=1}^{r+s} \in F(T)\text.\]
Since $F \subseteq \Sigma$ and $d_1 + \cdots + d_{r+s} = d$, this is equivalent to
\[(d_i \log \max_{1\leq j\leq n}\{|z_{i,j}|\})_{i=1}^{r+s} \in F(\infty)\ \text{ and }\ \prod_{i=1}^{r+s} \max_{1\leq j\leq n}\{|z_{i,j}|\}^{d_i} \leq T^d\text.\]
The set $S_F^n(\infty)$ is defined similarly. Here are some basic properties of $S_F^n(T)$:  
\begin{enumerate}[(i)]
 \item $S_F^n(T) = TS_F^n(1)$ is homogeneously expanding.
 \item $S_F^n(1)$ is bounded.
 \item $\partial S_F^n(1) \in \Lip(nd, M_n, L_n)$ for some $M_n$, $L_n$.
 \item $S_F^n(1)$ is measurable and $\Vol S_F^n(1) = n^q 2^{nr}\pi^{ns}R_K$. 
\end{enumerate}
Properties (i), (ii) follow directly from the definition, and (iii), (iv) are immediate consequences of \cite[Lemma 3, Lemma 4]{Masser2006}. Strictly speaking, the case $n = 1$ is not covered by \cite{Masser2006}, but the proofs remain correct without change. We need a slightly modified version: Define 
\begin{equation}\label{def_SFN*}
S_F^{n*}(T) := S_F^n(T) \cap ((\R^\times)^{nr} \times (\C^\times)^{ns})\text.
\end{equation}
Then (i) -- (iv) hold as well for $S_F^{n*}(T)$. This is clear for (i), (ii), (iv). For (iii), let $X := (\R^{nr}\times \C^{ns})\setminus((\R^\times)^{nr} \times (\C^\times)^{ns})$. Then $\partial S_F^{n*}(1) \subseteq \partial S_F^n(1) \cup (\overline{S_F^n(1)} \cap X)$. Since $\overline{S_F^n(1)}$ is bounded and $X$ is a union of finitely many proper subspaces, we have $(\overline{S_F^n(1)} \cap X) \in  \Lip(nd, M'_n, L'_n)$, for suitably chosen $M'_n$, $L'_n$, so
\[\partial S_F^{n*}(1) \in \Lip(nd, M_n + M'_n, \max\{L_n, L'_n\})\text.\]

\section{Proof of Lemma 3.1}\label{sec_proof_lemma}
Whenever we use Vinogradov's $\ll$ notation, the implicit constant may depend on $K$. Let us start by summing over $y_1$, $y_2$, $y_3$, for fixed $y_{jk}$, $y_{kj}$. Write 
\[V' := V\setminus \{1,2,3\} = \{12,21,23,32,31,13\}\text.\]
For any choice of $y_v$, $v \in V'$, we define $\xi_j := y_{jk}y_{jl}y_{kj}^2y_{lj}^2$. The height condition in \eqref{def_MRB} implies that
\[|N(y_j)^3 N(\xi_j)| = \prod_{i=1}^{r+s} |\sigma_i(y_j^3 \xi_j)|^{d_i} \leq B\text.\]
For $y_j \in \ida_j$, we obtain $|N(\xi_j)| \leq B |N(y_j)|^{-3} \leq B \Norm\ida_j^{-3}$. By our choice of $\mathcal{R}$ in \ref{subsection_fundamental_domain}, we can write the sum in Lemma \ref{main_lemma} as
\begin{align}\label{split_sum}
\sum_{\substack{{\bf y} \in \mathcal{R}(B)\\y_v \in \ida_v}}1 &= \sum_{\substack{(y_v)_{v \in V'} \in \mathcal{R}_2\\y_v \in \ida_v\\\forall j\ : \ |N(\xi_j)|\leq B\Norm\ida_j^{-3}}}\ \ \sum_{\substack{(y_1, y_2, y_3) \in \mathcal{R}_1\\y_j \in \ida_j\\\prod\limits_{i=1}^{r+s} \max\limits_j\{|\sigma_i(y_j^3 \xi_j)|\}^{d_i} \leq B}} 1.
\end{align}

\subsection{The first summation}
Here, we handle the inner sum in \eqref{split_sum}. The necessary tool is provided in Lemma \ref{lemma_y1y2y3}.
\begin{lemma}\label{det_succ_minima}
Let $\ida$ be a fractional ideal of $K$, and let $\tau$ be the linear automorphism of $\R^r\times\C^s$ (regarded as $\R^d$) given by $\tau(z_1, \ldots, z_{r+s}) = (t_1 z_1, \ldots, t_{r+s} z_{r+s})$, with $t_1$, $\ldots$, $t_{r+s} > 0$. Let $\sigma : K \to \R^r \times \C^s$ be the standard embedding. Then ${\tau} \circ {\sigma}(\ida)$ is a lattice in $\R^r \times \C^s$ of determinant
\[\det ({\tau} \circ {\sigma}(\ida)) = t_1^{d_1} \cdots t_{r+s}^{d_{r+s}}\cdot 2^{-s}\cdot \Norm(\ida_j)\cdot \sqrt{|\Delta_K|}\]
and first successive minimum $\lambda \geq (t_1^{d_1} \cdots t_{r+s}^{d_{r+s}}\cdot \Norm\ida)^{1/d}$.
\end{lemma}

\begin{proof}
For $d=1$, the lemma is trivial, so we assume $d \geq 2$. Classically, ${\sigma}(\ida)$ is a lattice in $\R^r \times \C^s$ of determinant $2^{-s} \Norm(\ida_j)\sqrt{|\Delta_K|}$. Since $\tau$ is a linear automorphism of determinant $t_1^{d_1} \cdots t_{r+s}^{d_{r+s}}$, it follows immediately that ${\tau} \circ {\sigma}(\ida)$ is a lattice with the correct determinant.

For $\lambda$, we slightly generalize the argument in \cite[Lemma 5]{Masser2006} (see also \cite[Lemma 9.7]{Widmer2010}). There is an $\alpha \in \ida$ with $\lambda = |{\tau} \circ {\sigma}(\alpha)|$. By the inequality of weighted arithmetic and geometric means, we have
\begin{align*}
\lambda^2 = \sum_{i=1}^{r+s}|t_i \alpha^{(i)}|^2 &\geq \frac{1}{2}\sum_{i=1}^{r+s}d_i |t_i\alpha^{(i)}|^2
\geq\frac{d}{2}\left(\prod_{i=1}^{r+s}|t_i \alpha^{(i)}|^{d_i} \right)^\frac{2}{d} \geq(t_1^{d_1} \cdots t_{r+s}^{d_{r+s}} |N(\alpha)|)^\frac{2}{d}\text.
\end{align*}
The lemma follows upon noticing that $|N(\alpha)|\geq\Norm\ida$.
\end{proof}

\begin{lemma}\label{lemma_y1y2y3}
Given constants $C_{ij} > 0$, for $i \in \{1, \ldots, r+s\}$ and $j \in \{1, 2, 3\}$, let
\[C_j := C_{1j}^{d_1} \cdots C_{r+s,j}^{d_{r+s}}\text.\]
Let $\ida_1$, $\ida_2$, $\ida_3 \neq \{0\}$ be fractional ideals of $K$, and $\mathcal{R}_1$ a system of representatives for the orbits of $(K^\times)^3$ under the action of $\mathcal{F}$ by scalar multiplication. Define
\[M_1(T) := (\ida_1 \times \ida_2 \times \ida_3) \cap \{(y_1, y_2, y_3) \in \mathcal{R}_1 \mid \prod_{i=1}^{r+s}\max_{1\leq j\leq 3}\{C_{ij} |y_j^{(i)}|\}^{d_i} \leq T^d\}\text.\]
Then $M_1(T)$ is finite and
\[|M_1(T)| = \frac{3^q2^{3r} (2\pi)^{3s} R_K}{(\sqrt{|\Delta_k|})^3 C_1 C_2 C_3 \Norm\ida_1\Norm\ida_2\Norm\ida_3} T^{3d} + O\left(\frac{T^{3d-1}\max_j\{C_j\Norm\ida_j\}^{1/d}}{C_1C_2C_3\Norm\ida_1\Norm\ida_2\Norm\ida_3}\right)\text,\]
for all $T > 0$. The implicit $O$-constant depends only on $K$.
\end{lemma}

\begin{proof}
We notice that $|M_1(T)|$ does not depend on the choice of $\mathcal{R}_1$, since both $\ida_1 \times \ida_2 \times \ida_3$ and the height condition are invariant under scalar multiplication of $(y_1, y_2, y_3)$ by units. Hence, it is enough to prove the lemma with a specific choice of $\mathcal{R}_1$, which we construct below.

Let $\sigma : K^3 \to \R^{3r} \times \C^{3s}$ be the embedding given by $\sigma({\bf y}) = (\sigma_i({\bf y}))_{i=1}^{r+s}$. For $i \in \{1, \ldots, r+s\}$, let $\phi_i$ be the linear automorphism of $\R^3$ (if $i \leq r$) or $\C^3$ (if $i > r$) given by $\phi_i(z_1, z_2, z_3) = (C_{i1}z_1, C_{i2}z_2, C_{i3}z_3)$, and let $\phi : \R^{3r}\times \C^{3s} \to \R^{3r}\times \C^{3s}$ be the automorphism obtained by applying the $\phi_i$ component-wise. 

With $S_F^{3*}(T)$ as in \eqref{def_SFN*}, we define $\mathcal{R}_1$ as the set of all ${\bf y} \in (K^\times)^3$ such that $\phi \circ \sigma({\bf y}) \in S_F^{3*}(\infty)$. Then $\mathcal{R}_1$ is a system of representatives for the orbits of $(K^\times)^3$ under the action of $\mathcal{F}$ by scalar multiplication. Indeed, for any ${\bf y} \in (K^\times)^3$ and $\zeta \in \mathcal{F}$, we have
\[(d_i \log \max_{1 \leq j \leq 3}\{|C_{ij} \sigma_i(\zeta y_j)|\})_{i=1}^{r+s} = (d_i \log \max_{1 \leq j \leq 3}\{|C_{ij} \sigma_i(y_j)|\})_{i=1}^{r+s} + l(\zeta)\text,\]
and $F(\infty)$ is a system of representatives for the orbits of the additive action of $l(\mathcal{F})$ on $\R^{r+s}$. 

Let $\Lambda := \phi \circ \sigma(\ida_1 \times \ida_2 \times \ida_3)$. Then $\Lambda$ is a lattice in $\R^{3r} \times \C^{3s}$, and $\phi\circ\sigma$ induces a one-to-one correspondence between $M_1(T)$ and $\Lambda \cap S_F^{3*}(T)$. Therefore,
\begin{equation}\label{count_lattice_pts}
|M_1(T)| = |\Lambda \cap S_F^{3*}(T)|\text.
\end{equation}
Since $S_F^{3*}(T)$ is bounded, $M_1(T)$ is finite. To simplify the notation, we change the order of coordinates by
\[(z_{11}, z_{12}, z_{13}, \ldots, z_{r+s,1}, z_{r+s,2}, z_{r+s,3}) \mapsto (z_{11}, \ldots, z_{r+s,1}, \ldots, z_{13}, \ldots, z_{r+s,3})\text.\]
This way, $\R^{3r}\times\C^{3s}$ becomes $(\R^r \times \C^s)^3$, and $\Lambda$ becomes
\[\Lambda = \tau_1 \circ \sigma(\ida_1) \times \tau_2 \circ \sigma(\ida_2) \times \tau_3 \circ \sigma(\ida_3)\text,\]
where $\sigma : K \to \R^r \times \C^s$ is the standard embedding given by $\sigma(y) = (\sigma_i(y))_{i=1}^r$ and 
\[\tau_j(z_1, \ldots, z_{r+s}) := (C_{1j}z_1, \ldots, C_{r+s,j}z_{r+s})\text.\]
Each $\Lambda_j := \tau_j\circ\sigma(\ida_j)$ is a lattice in $\R^r \times \C^s = \R^d$. Let $\lambda_j$ be the first successive minimum of $\Lambda_j$. By Lemma \ref{det_succ_minima}, we have
\[\det\Lambda = \det\Lambda_1\cdot\det\Lambda_2\cdot\det\Lambda_3 = 2^{-3s}(\sqrt{|\Delta_K|})^3C_1C_2C_3\Norm\ida_1\Norm\ida_2\Norm\ida_3\]
and
\[\lambda_j \geq (C_j \Norm\ida_j)^{1/d}\text.\]
The lemma now follows from \eqref{count_lattice_pts}, Lemma \ref{lattice_pt_homog_exp} and the facts from \ref{subsection_basic_set}.
\end{proof}

The inner sum in \eqref{split_sum} is exactly $|M_1(T)|$ in Lemma \ref{lemma_y1y2y3}, with 
\[C_{ij} := |\sigma_i(\xi_j)|^{1/3}\text,\quad C_j := |N(\xi_j)|^{1/3}\quad\text{, and }\quad T := B^{1/(3d)}\text.\]
Observe that $C_1 C_2 C_3 = |N(\xi_1 \xi_2 \xi_3)|^{1/3} = \prod_{v \in V'}|N(y_v)|$. We define
\begin{align}\label{main_term}
\mathcal{M}(B, (\ida_v)_{v}) &:=  \sum_{\substack{(y_v)_{v \in V'} \in \mathcal{R}_2\\y_v \in \ida_v\\\forall j\ : \ |N(\xi_j)|\leq B\Norm\ida_j^{-3}}} \frac{1}{\prod_{v \in V'}|N(y_v)|}\text{, and}\\
\label{error_term}
\mathcal{R}(B, (\ida_v)_{v}) &:= \sum_{\substack{(y_v)_{v \in V'} \in \mathcal{R}_2\\y_v \in \ida_v\\\forall j\ : \ |N(\xi_j)|\leq B\Norm\ida_j^{-3}}} \frac{\max_j\{|N(\xi_j)|\}^{1/(3d)}}{\prod_{v \in V'}|N(y_v)|}\text.
\end{align}
Then \eqref{split_sum} and Lemma \ref{lemma_y1y2y3} imply
\begin{equation}\label{mainpluserror}
\begin{aligned}
\sum_{\substack{{\bf y} \in \mathcal{R}(B)\\y_v \in \ida_v}}1 &= \frac{3^q 2^{3r}(2\pi)^{3s} R_K B}{(\sqrt{|\Delta_K|})^3\Norm\ida_1\Norm\ida_2\Norm\ida_3} \mathcal{M}(B, (\ida_v)_{v})\\
&+ O\left(\frac{\max_j\{\Norm\ida_j\}^{1/d}}{\Norm\ida_1\Norm\ida_2\Norm\ida_3} B^{1-1/(3d)}\mathcal{R}(B, (\ida_v)_{v})\right)\text.
\end{aligned}
\end{equation}
Recall that the $\Norm\ida_v$ are bounded from below by a positive constant $c$ depending only on $K$. This implies, for example,
\begin{equation}\label{error_term_norm_bound}\Norm(\ida_{jk}\ida_{jl}\ida_{kj}^2\ida_{lj}^2)^{1/(3d)} \ll \prod_{v \in V'}\Norm\ida_v^{2/(3d)}\end{equation}
and
\begin{equation}\label{error_term_norm_bound_2}\Norm(\ida_j^3\ida_{jk}\ida_{jl}\ida_{kj}^2\ida_{lj}^2)^{-1} \leq c_2\text,\end{equation}
for some constant $c_2 \geq 1$ depending only on $K$.

\subsection{The error term}
With $\mathcal{R}_2$ as in Lemma \ref{lemma_R}, the term $\mathcal{R}(B, (\ida_v)_{v})$ has the form
\[\mathcal{R}(B, (\ida_v)_{v}) = \sum_{\rho \in R_\mathcal{F}}\sum_{\substack{\forall v \neq 12 \ : \ y_{v} \in R \cap \ida_{v}\\y_{12} \in \rho R \cap \ida_{12}\\\forall j\ : \ |N(\xi_j)|\leq B\Norm\ida_j^{-3}}}\frac{\max_j\{|N(\xi_j)|\}^{1/(3d)}}{\prod_{v \in V'}|N(y_v)|}\text.\]
Both $R$ and $\rho R$ are systems of representatives for $K^\times / \mathcal{F}$, so they contain exactly $\omega_K$ generators for every nonzero principal fractional ideal of $K$. Let $H_v$ be the principal fractional ideal $H_v = y_v \O_K$. The norm condition and the summand in the inner sum depend only on $(H_v)_{v \in V'}$. Therefore, the sum does not depend on $\rho$. Since $|\mathcal{R}_\mathcal{F}| = 3^q \ll 1$, we obtain
\[\mathcal{R}(B, (\ida_v)_{v}) \ll \sum_{\substack{\{0\}\neq H_v \in P_K\text{, } v \in V'\\H_v \subseteq \ida_v\\\forall j:\ \Norm(H_{jk}H_{jl}H_{kj}^2 H_{lj}^2)\leq B\Norm\ida_j^{-3}}}\frac{\max_j\{\Norm(H_{jk}H_{jl}H_{kj}^2 H_{lj}^2)\}^{1/(3d)}}{\prod_{v \in V'}\Norm(H_v)}\text.\]
We replace $H_v$ by $H_v \ida_v^{-1} \unlhd \O_K$ and use \eqref{error_term_norm_bound}, \eqref{error_term_norm_bound_2} to bound this sum by
\[\ll \frac{1}{\prod\limits_{v \in V'}\Norm(\ida_v)^{1-2/(3d)}}\sum_{\substack{\{0\}\neq H_v \unlhd \O_K\text{, } v \in V'\\H_v \in [\ida_v]^{-1}\\\forall j:\ \Norm(H_{jk}H_{jl}H_{kj}^2 H_{lj}^2)\leq c_2B}}\frac{\max_j\{\Norm(H_{jk}H_{jl}H_{kj}^2 H_{lj}^2)\}^{1/(3d)}}{\prod_{v \in V'}\Norm(H_v)}\text.\]
Let us denote the above sum by $\mathcal{R}_1(B, (\ida_v)_v)$. What follows is a rather straightforward generalization of arguments used by Heath-Brown and Moroz \cite{Heath-Brown1999} and Derenthal and Janda \cite{Derenthal2011}. By symmetry, we may assume that the maximum in the summand is taken for $j = 1$. This allows us to bound $\mathcal{R}_1(B, (\ida_v)_v)$ by
\begin{align*}&\ll \sum_{\substack{\{0\}\neq H_v \unlhd \O_K\text{, }v \in V'\\\forall j:\ \Norm(H_{jk}H_{jl}H_{kj}^2 H_{lj}^2)\leq c_2B}}\frac{1}{\Norm(H_{12}H_{13})^{1-1/(3d)}\Norm(H_{21}H_{31})^{1-2/(3d)}\Norm(H_{23}H_{32})}\\
&\ll \sum_{\substack{\{0\}\neq H_{ij} \unlhd \O_K,\ i\neq 1  \\ \Norm H_{ij}\leq c_2B}}\frac{1}{\Norm(H_{21}H_{31})^{1-2/(3d)}\Norm(H_{23}H_{32})}\sum_{\substack{\{0\}\neq U \unlhd \O_K\\ \Norm U\leq u}}\frac{d(U)}{\Norm U^{1 - 1/(3d)}}\text,\end{align*}
where $u := c_2B\Norm(H_{21}H_{31})^{-2}$ and $d$ is the divisor function for nonzero ideals. 

\begin{lemma}\label{abel_summation}
For $T \geq 1$, we have
\[\sum_{\substack{\{0\}\neq\ida \unlhd \O_K\\ \Norm\ida \leq T}} \Norm\ida^\alpha \ll \begin{cases}T^{\alpha + 1}, &\text{ if }-1 < \alpha \leq 0\\\max\{1,\log T\}, &\text{ if }\alpha = -1\text.\end{cases}\]
\end{lemma}

\begin{proof}
This is a straightforward generalization of \cite[Lemma 4]{Derenthal2011}. The proof uses Abel's summation formula and the well known fact that 
\[|\{\{0\} \neq \ida \unlhd \O_K \mid \Norm\ida \leq T\}| \ll T\text.\] 
\end{proof}
In the following computation, the sums run over nonzero ideals of $\O_K$. Using Lemma \ref{abel_summation}, we obtain 
\begin{align*}
\sum_{\Norm U\leq u}\frac{d(U)}{\Norm U^{1 - 1/(3d)}} &= \sum_{\Norm U\leq u}\sum_{V \mid U}\Norm U^{-1+1/(3d)}\\
&= \sum_{\Norm V \leq u}\Norm V^{-1+1/(3d)}\sum_{\Norm U \leq u/\Norm V}\Norm U^{-1+1/(3d)}\\
&\ll \sum_{\Norm V \leq c_2 B} \Norm V^{-1+1/(3d)} (u/\Norm V)^{1/(3d)}\ll u^{1/(3d)} \log B\text.
\end{align*}
Therefore,
\begin{align*}\mathcal{R}_1(B, (\ida_v)_v) &\ll B^{1/(3d)}\log B \sum_{\substack{\{0\}\neq H_{ij} \unlhd \O_K,\ i\neq 1  \\ \Norm H_{ij}\leq c_2B}}\frac{1}{\Norm(H_{21}H_{31}H_{23}H_{32})}\\
&\ll B^{1/(3d)}(\log B)^5\text.  
\end{align*}
Having estimated $\mathcal{R}_1(B, (\ida_v)_v)$ and thus $\mathcal{R}(B, (\ida_v)_v)$, we obtain from \eqref{mainpluserror}:
\begin{equation}\label{only_main_term_left}
\begin{aligned}
\sum_{\substack{{\bf y} \in \mathcal{R}(B)\\y_v \in \ida_v}}1 &= \frac{3^q 2^{3r}(2\pi)^{3s} R_K B}{(\sqrt{|\Delta_K|})^3\Norm\ida_1\Norm\ida_2\Norm\ida_3} \mathcal{M}(B, (\ida_v)_{v})\\
&+ O\left(\frac{\max_j\{\Norm\ida_j\}^{1/d}}{\prod_j\Norm\ida_j \prod_{i\neq j}\Norm\ida_{ij}^{1-2/(3d)}} B (\log B)^5\right)\text.
\end{aligned}
\end{equation}

\subsection{The main term}
Just as before, we have
\[\mathcal{M}(B, (\ida_v)_{v}) = \sum_{\rho \in R_\mathcal{F}}\sum_{\substack{\forall v \neq 12 \ : \ y_{v} \in R \cap \ida_{v}\\y_{12} \in \rho R \cap \ida_{12}\\\forall j\ : \ |N(\xi_j)|\leq B\Norm\ida_j^{-3}}}\frac{1}{\prod_{v \in V'}|N(y_v)|}\text.\]
For all $v \in V'$, let $\idb_v \in \mathcal{C}$ with $[\idb_v] = [\ida_v]$, and $t_v \in K^\times$ with $t_v \ida_v = \idb_v$. Moreover, we define $b_j := \Norm(\ida_j^{3}\ida_{jk}\ida_{jl}\ida_{kj}^2\ida_{lj}^2)^{-1}\Norm(\idb_{jk}\idb_{jl}\idb_{kj}^2\idb_{lj}^2)$. Then \eqref{error_term_norm_bound_2} implies that 
\begin{equation}\label{error_term_norm_bound_3}
b_j \leq c_3\text{, for all }j \in \{1, 2, 3\}\text,
\end{equation}
with a constant $c_3 \geq 1$ depending only on $K$. We replace $y_v$ by $t_v y_v$ and obtain
\[\mathcal{M}(B, (\ida_v)_{v}) = \left(\prod_{v \in V'}\frac{\Norm\idb_v}{\Norm\ida_v}\right)\sum_{\rho \in \mathcal{R}_\mathcal{F}}\quad\sum_{\substack{\forall v \neq 12 \ : \ y_{v} \in t_v R \cap \idb_{v}\\y_{12} \in t_v \rho R \cap \idb_{12}\\\forall j\ : \ |N(\xi_j)|\leq b_j B}}\frac{1}{\prod_{v \in V'}|N(y_v)|}\text.\]
Again, the inner sum does not depend on the sets of representatives $t_v R$, $t_v \rho R$ for $K^\times / \mathcal{F}$. Thus, 
\begin{equation}\label{M_independent_of_fundamental_system}
\mathcal{M}(B, (\ida_v)_{v}) = 3^{q}\left(\prod_{v \in V'}\frac{\Norm\idb_v}{\Norm\ida_v}\right) \sum_{\substack{y_{v} \in R \cap \idb_{v}\text{, }v \in V'\\\forall j\ : \ |N(\xi_j)|\leq b_j B}}\frac{1}{\prod_{v \in V'}|N(y_v)|}\text,
\end{equation}
where $R$ is any system of representatives for $K^\times / \mathcal{F}$. Let $\sigma : K \to \R^r \times \C^s$ be the standard embedding, and let $S_F^1(T)$ be defined as in \ref{subsection_basic_set}. We choose $R$ to be the set of all $y \in K^\times$ with $\sigma(y) \in S_F^1(\infty)$. This is indeed a set of representatives for $K^\times / \mathcal{F}$: For any $y \in K^\times$, $\zeta \in \mathcal{F}$, we have
\[(d_i \log |\sigma_i(\zeta y)|)_{i=1}^{r+s} = (d_i \log |\sigma_i(y)|)_{i=1}^{r+s} + l(\zeta)\text,\]
and $F(\infty)$ is a system of representatives for the orbits of the additive action of $l(\mathcal{F})$ on $\R^{r+s}$. We will first consider the sum
\[\mathcal{M}_1(B, (\idb_v)_{v}) := \sum_{\substack{y_{v} \in R \cap \idb_{v}\text{, }v \in V'\\\forall j\ : \ |N(\xi_j)|\leq B}}\frac{1}{\prod_{v \in V'}|N(y_v)|}\text.\]
For any ${\bf z} \in \R^r \times \C^s$, let $N({\bf z}) := |z_1|^{d_1} \cdots |z_{r+s}|^{d_{r+s}}$. We define $M(B)$ as the set of all $({\bf z}_v)_{v \in V'} \in (\R^r \times \C^s)^6$ such that
\[\text{for all } v \in V'\text{, we have }{\bf z}_v \in S_F^1(\infty) \text{ and } N({\bf z}_v) \geq 1\text{,}\]
and
\[\text{for all } j\text{, we have }N({\bf z}_{jk})N({\bf z}_{jl})N({\bf z}_{kj})^2N({\bf z}_{lj})^2 \leq B\text.\]
Then $M(B)$ is bounded for all $B$. Let $\Lambda$ be the lattice in $(\R^r \times \C^s)^6$ defined by
\[\Lambda := \prod_{v \in V'} \sigma(\idb_v)\text.\]
By the component-wise extension of $\sigma$ to $K^6$, we obtain 
\begin{equation}\label{M1_first_equality}
\mathcal{M}_1(B, (\idb_v)_{v}) = \sum_{({\bf z}_v)_v \in \Lambda \cap M(B)}\frac{1}{\prod_{v \in V'}N({\bf z}_v)}\text.
\end{equation}
We identify $\C$ with $\R^2$ and estimate this sum by an integral. Let
\[I(B) := \left(\frac{2^s}{\sqrt{|\Delta_K|}}\right)^6 \frac{1}{\prod_{v \in V'}\Norm\idb_{v}}\int_{M(B)}\prod_{v \in V'}\frac{d{\bf z_v}}{N({\bf z}_v)}\text.\]
\begin{lemma}\label{lemma_sum_integral}
We have
\[\sum_{({\bf z}_v)_v \in \Lambda \cap M(B)}\frac{1}{\prod_{v \in V'}N({\bf z}_v)} = I(B) + O((\log B)^5)\text,\]
for $B \geq e$. The implicit $O$-constant depends on $K$.
\end{lemma}

\begin{proof}
This is a generalization of \cite[Lemma 5]{Derenthal2011}. Let us fix some notation. For $v \in V'$, let $F_v$ be a fundamental parallelotope for the lattice $\sigma(\idb_v) \subseteq \R^r \times \C^s = \R^d$, and let $R_v$ be the minimal $d$-dimensional interval containing $F_v$. We denote the side lengths of $R_v$ by $l_{v,1}$, $\ldots$, $l_{v,d}$.
For any ${\bf z} = (z_1, \ldots, z_d) \in \R^d$ satisfying 
\begin{equation}\label{lemma_sum_integral_away_from_axes}
|z_i| \geq 1 + l_{v, i} \text{ for all }i \in \{1, \ldots, d\}\text, 
\end{equation}
let $R_v({\bf z})$ be the (unique) translate of $R_v$ such that ${\bf z}$ is the corner of $R_v({\bf z})$ at utmost distance from the origin, and let $F_v({\bf z})$ be the (unique) translate of $F_v$ contained in $R_v({\bf z})$. Similarly, for any ${\bf z}$ with
\begin{equation}\label{lemma_sum_integral_away_from_axes_2}
|z_i| \geq 1 \text{ for all }i \in \{1, \ldots, d\}\text, 
\end{equation}
let $R_v'({\bf z})$ be the (unique) translate of $R_v$ such that ${\bf z}$ is the corner of $R_v'({\bf z})$ closest to the origin, and let $F_v'({\bf z})$ be the (unique) translate of $F_v$ contained in $R_v'({\bf z})$. Consistently with the above definition of $N({\bf z})$ for ${\bf z} \in \R^r\times\C^s$, we let
\[N({\bf z}) := |z_1\cdots z_r(z_{r+1}^2 + z_{r+2}^2)\cdots (z_{d-1}^2 + z_d^2)|\text.\]
Since $N({\bf z}) \geq N({\bf y})$ for all ${\bf y} \in F_v({\bf z})$, we have
\begin{equation}\label{sum_smaller_than_integral}
 \frac{1}{N({\bf z})} \leq \frac{1}{\Vol F_v({\bf z})}\int_{F_v({\bf z})}\frac{d{\bf y}}{N({\bf y})} = \frac{2^s}{\sqrt{|\Delta_K|}\Norm\idb_v} \int_{F_v({\bf z})}\frac{d{\bf y}}{N({\bf y})}\text.
\end{equation}
Similarly,
\begin{equation}\label{integral_smaller_than_sum}
 \frac{1}{N({\bf z})} \geq \frac{1}{\Vol F_v'({\bf z})}\int_{F_v'({\bf z})}\frac{d{\bf y}}{N({\bf y})} = \frac{2^s}{\sqrt{|\Delta_K|}\Norm\idb_v} \int_{F_v'({\bf z})}\frac{d{\bf y}}{N({\bf y})}\text.
\end{equation}
Clearly, if ${\bf z} \neq {\bf z'} \in \sigma(\idb_v)$ with \eqref{lemma_sum_integral_away_from_axes} then $F_v({\bf z}) \cap F_v({\bf z'}) = \emptyset$. Let us first prove that
\begin{equation}\label{sum_smaller_int}\sum_{({\bf z}_v)_v \in \Lambda \cap M(B)}\frac{1}{\prod_{v \in V'}N({\bf z}_v)} \leq I(B) + O((\log B)^5)\text.\end{equation}
To this end, we define
\[E(B) := \{({\bf z}_v)_v \in M(B) \mid \text{ all }{\bf z}_v\text{ satisfy \eqref{lemma_sum_integral_away_from_axes} and }F_v({\bf z}_v) \subseteq S_F^1(\infty)\}\text,\]
and $G(B) := M(B) \setminus E(B)$. Keep in mind that $E(B)$ and $G(B)$ depend on $(\idb_v)_{v\in V'}$. For any $({\bf z}_v)_v \in \Lambda \cap E(B)$, we have $\prod_v F_v({\bf z}_v) \subseteq M(B)$. Therefore,
\begin{align*}&\sum_{({\bf z}_v)_v \in \Lambda \cap E(B)}\frac{1}{\prod_{v \in V'}N({\bf z}_v)} \leq \sum_{({\bf z}_v)_v \in \Lambda \cap E(B)}\prod_{v\in V'} \frac{2^s}{\sqrt{|\Delta_K|}\Norm\idb_v} \int_{F_v({\bf z}_v)}\frac{d{\bf y}}{N({\bf y})}\\
 &\quad \leq \left(\frac{2^s}{\sqrt{|\Delta_K|}}\right)^6 \frac{1}{\prod_{v \in V'}\Norm\idb_{v}}\sum_{({\bf z}_v)_v \in \Lambda \cap E(B)}\prod_{v \in V'}\int_{F_v({\bf z}_v)}\frac{d{\bf z_v}}{N({\bf z}_v)}\leq I(B)\text.
\end{align*}
We need to prove that
\begin{equation}\label{sum_smaller_int_error_term}
\sum_{({\bf z}_v)_v \in \Lambda \cap G(B)}\frac{1}{\prod_{v \in V'}N({\bf z}_v)} = O((\log B)^5)\text.
\end{equation}
For every $({\bf z}_v)_v \in \Lambda \cap G(B)$, there is at least one $w \in V'$ such that either 
\begin{equation}\label{sum_smaller_int_error_case_1_start}
{\bf z}_{w} \text{ does not satisfy \eqref{lemma_sum_integral_away_from_axes}}
\end{equation}
or
\begin{equation}\label{sum_smaller_int_error_case_2_start}
{\bf z}_{w} \text{ satisfies \eqref{lemma_sum_integral_away_from_axes} and } F_w({\bf z}_w)\not\subseteq S_F^1(\infty)\text.
\end{equation}
Therefore, we have
\begin{equation}\label{sum_smaller_int_split_up}
\begin{aligned}
&\sum_{({\bf z}_v)_v \in \Lambda \cap G(B)}\frac{1}{\prod_{v \in V'}N({\bf z}_v)} \leq \sum_{w \in V'}\sum_{\substack{({\bf z}_v)_v \in \Lambda \cap S_F^1(\infty)^6\\N({\bf z_v})\leq B\\\eqref{sum_smaller_int_error_case_1_start}\text{ or }\eqref{sum_smaller_int_error_case_2_start}}}\frac{1}{\prod_{v \in V'}N({\bf z}_v)}\\
&=\sum_{w \in V'}\left(\prod_{v \neq w}\sum_{\substack{{\bf z} \in \sigma(\idb_v) \cap S_F^1(\infty)\\N({\bf z})\leq B}}\frac{1}{N({\bf z})}\right)\sum_{\substack{{\bf z}\in \sigma(\idb_w)\cap S_F^1(\infty)\\N({\bf z})\leq B\\\eqref{sum_smaller_int_error_case_1_start}\text{ or }\eqref{sum_smaller_int_error_case_2_start}\text{ for }{\bf z}}}\frac{1}{N({\bf z})}\text.
\end{aligned}
\end{equation}
Now
\begin{equation}\label{sum_smaller_int_easy_part}
\sum_{\substack{{\bf z} \in \sigma(\idb_v) \cap S_F^1(\infty)\\N({\bf z})\leq B}}\frac{1}{N({\bf z})} = \omega_K \sum_{\substack{\{0\} \neq H \in P_K\\ H \subseteq \idb_v\\\Norm H \leq B}}\frac{1}{\Norm H}\leq\sum_{\substack{\{0\}\neq H \unlhd \O_K\\\Norm H \leq B}}\frac{1}{\Norm H} \ll \log B\text,
\end{equation}
by Lemma \ref{abel_summation}. Moreover, we write
\begin{equation}\label{sum_smaller_int_abel_sum}
\sum_{\substack{{\bf z}\in \sigma(\idb_w)\cap S_F^1(\infty)\\N({\bf z})\leq B\\\eqref{sum_smaller_int_error_case_1_start}\text{ or }\eqref{sum_smaller_int_error_case_2_start}\text{ for }{\bf z}}}\frac{1}{N({\bf z})} = \sum_{n=1}^B a_n \cdot \frac{1}{n}\text,
\end{equation}
with $a_n := |\{{\bf z} \in \sigma(\idb_w) \cap S_F^1(\infty) \mid N({\bf z}) = n\text{, }\eqref{sum_smaller_int_error_case_1_start}\text{ or }\eqref{sum_smaller_int_error_case_2_start}\text{ holds for }{\bf z}\}|$. We will apply the Abel sum formula, so we need to understand
\[A(T) := \sum_{n \leq T}a_n = |\{{\bf z} \in \sigma(\idb_w) \cap S_F^1(T^{1/d}) \mid \eqref{sum_smaller_int_error_case_1_start}\text{ or }\eqref{sum_smaller_int_error_case_2_start}\text{ holds for }{\bf z}\}|\text.\]
Let 
\begin{equation}\label{def_H}
H := \{{\bf z} \in \R^d \mid z_1 \cdots z_d = 0\}\text,
\end{equation}
and let $D_w$ be the $d$-dimensional interval
\begin{equation}\label{def_Dw}
D_w := [-(l_{w,1} + 1), l_{w,1} + 1] \times \cdots \times [-(l_{w,d} + 1), l_{w,d} + 1] \subseteq \R^d\text.
\end{equation}
Then any ${\bf z}$ counted by $A(T)$ satisfies $({\bf z} + D_w) \cap H \neq \emptyset$ (if \eqref{sum_smaller_int_error_case_1_start} holds) or ${\bf z} + D_w \not \subseteq S_F^1(T^{1/d})$ (if \eqref{sum_smaller_int_error_case_2_start} holds). Therefore, any such ${\bf z}$ is contained in $A_1(T) \cup A_2(T)$, where
\begin{align*}
A_1(T) :&= \{{\bf z} \in \sigma(\idb_w) \mid ({\bf z} + D_w) \cap \partial S_F^1(T^{1/d}) \neq \emptyset\}\\
&\supseteq\{{\bf z} \in \sigma(\idb_w) \cap S_F^1(T^{1/d}) \mid ({\bf z} + D_w) \not\subseteq S_F^1(T^{1/d})\}\text,
\end{align*}
and
\begin{align*}
A_2(T) :&= \{{\bf z} \in \sigma(\idb_w) \mid ({\bf z} + D_w) \cap (S_F^1(T^{1/d}) \cap H) \neq \emptyset\}\\
&\supseteq\{{\bf z} \in \sigma(\idb_w) \cap S_F^1(T^{1/d}) \mid ({\bf z} + D_w) \subseteq S_F^1(T^{1/d}) \text{, } ({\bf z}_w + D_w) \cap H \neq \emptyset\}\text.
\end{align*}
Now $\partial S_F^1(T^{1/d}) = T^{1/d} \partial S_F^1(1) \in \Lip(d, M_1, T^{1/d}L_1)$. We recall that $\idb_v \in \mathcal{C}$, so Lemma \ref{translate_of_fixed_set}, \emph{(i)}, implies that 
\[|A_1(T)| \ll M_1 (L_1 T^{1/d} + 1)^{d-1} \ll T^{(d-1)/d}\text{, for all }T \geq 1\text.\]  
Moreover, $S_F^1(T^{1/d}) \cap H = T^{1/d} (S_F^1(1) \cap H)$, and clearly $S_F^1(1) \cap H \in \Lip(d, \tilde M_1, \tilde L_1)$ for some $\tilde M_1$ and $\tilde L_1$. By Lemma \ref{translate_of_fixed_set}, \emph{(i)}, 
\[|A_2(T)| \ll \tilde M_1 (\tilde L_1 T^{1/d} + 1)^{d-1} \ll T^{(d-1)/d}\text{, for all }T \geq 1\text.\]  
Therefore, $A(T) \ll T^{(d-1)/d}$, for $T \geq 1$. The Abel sum formula yields 
\begin{equation*}
\sum_{n=1}^B a_n \cdot \frac{1}{n} = A(B)/B + \int_{t=1}^B A(t)/t^2 dt \ll B^{-1/d} + \int_{t=1}^B t^{-(1+1/d)}dt \ll 1\text.
\end{equation*}
With \eqref{sum_smaller_int_split_up}, \eqref{sum_smaller_int_easy_part}, \eqref{sum_smaller_int_abel_sum}, we see that \eqref{sum_smaller_int_error_term} holds, which finishes the proof of \eqref{sum_smaller_int}. Let us prove the other inequality, that is
\begin{equation}\label{int_smaller_sum}I(B) \leq \sum_{({\bf z}_v)_v \in \Lambda \cap M(B)}\frac{1}{\prod_{v \in V'}N({\bf z}_v)} + O((\log B)^5)\text.\end{equation}
For every $v \in V'$ and every ${\bf z} \in \R^d$ satisfying \eqref{lemma_sum_integral_away_from_axes}, there is a unique $\lambda_v({\bf z}) \in \sigma(\idb_v)$ with \eqref{lemma_sum_integral_away_from_axes_2} such that ${\bf z} \in F_v'(\lambda_v({\bf z}))$. 
In a similar way as above, we define
\[E'(B):=\{({\bf z}_v)_v \in M(B) \mid \text{ all } {\bf z}_v \text{ satisfy \eqref{lemma_sum_integral_away_from_axes} and } \lambda_v({\bf z}_v) \in S_F^1(\infty)\}\text,\]
and $G'(B) := M(B) \setminus E'(B)$. Both $E'(B)$ and $G'(B)$ are clearly measurable. For any $({\bf z}_v)_v$ in $E'(B)$, the point $(\lambda_v({\bf z}_v))_v$ is the unique element of $\Lambda \cap M(B)$ with ${\bf z}_v \in F_v'(\lambda_v({\bf z}_v))$ for all $v \in V'$. With this and \eqref{integral_smaller_than_sum}, we obtain
\begin{align}
\frac{2^{6s}}{(\sqrt{|\Delta_K|})^6\prod\limits_{v \in V'}\Norm\idb_{v}}\int\limits_{E'(B)}\prod_{v \in V'}\frac{d{\bf z_v}}{N({\bf z}_v)}&\leq
\sum_{\substack{(\lambda_v)_v\in\\ \Lambda \cap M(B)}}\prod_{v \in V'}\frac{2^s}{\sqrt{|\Delta_K|}\Norm\idb_v} \int\limits_{F_v'({\bf \lambda_v})}\frac{d{\bf z}}{N({\bf z})}\nonumber\\&\leq \sum_{(\lambda_v)_v \in \Lambda \cap M(B)}\frac{1}{\prod_{v \in V'}N(\lambda_v)}\text.\label{int_smaller_sum_easy_part}
\end{align}
We need to prove that
\begin{equation}\label{int_smaller_sum_hard_part}
 \left(\frac{2^s}{\sqrt{|\Delta_K|}}\right)^6 \frac{1}{\prod_{v \in V'}\Norm\idb_{v}}\int_{G'(B)}\prod_{v \in V'}\frac{d{\bf z_v}}{N({\bf z}_v)} = O((\log B)^5)\text.
\end{equation}
For every $({\bf z}_v)_v \in G'(B)$, there is some $w \in V'$ such that either
\begin{equation}\label{int_smaller_sum_error_case_1_start}
 {\bf z}_{w} \text{ does not satisfy \eqref{lemma_sum_integral_away_from_axes}}
\end{equation}
or
\begin{equation}\label{int_smaller_sum_error_case_2_start}
 {\bf z}_{w} \text{ satisfies \eqref{lemma_sum_integral_away_from_axes} and } \lambda_w({\bf z}_w) \notin S_F^1(\infty)\text.
\end{equation}
Similarly to \eqref{sum_smaller_int_split_up}, we obtain
\begin{equation}\label{int_smaller_sum_split_up}
\int\limits_{G'(B)}\prod_{v \in V'}\frac{d{\bf z_v}}{N({\bf z}_v)} \leq \sum_{w \in V'}\left(\prod_{v \neq w}\int\limits_{\substack{{\bf z} \in S_F^1(\infty)\\1 \leq N({\bf z}) \leq B}}\frac{d{\bf z}}{N({\bf z})}\right)\int_{\substack{{\bf z} \in S_F^1(\infty)\\1 \leq N({\bf z}) \leq B\\\text{\eqref{int_smaller_sum_error_case_1_start} or \eqref{int_smaller_sum_error_case_2_start} for }{\bf z}}}\frac{d{\bf z}}{N({\bf z})}\text.
\end{equation}
We denote the Lebesgue measure on $\R$, $\R^d$ by $m_1$, $m_d$. The restriction of $N$ to $S_F^1(\infty)$ defines a measurable function $N_1 : S_F^1(\infty) \to \R$. Since 
\[(m_d \circ N_1^{-1})((a,b]) = \Vol S_F^1(b^{1/d}) - \Vol S_F^1(a^{1/d}) = (b-a) \Vol S_F^1(1)\text,\]
for all $0 < a \leq b \in \R$, we obtain $m_d \circ N_1^{-1} = \Vol S_F^1(1) m_1$ on $\R^{>0}$. Therefore, 
\begin{equation}\label{integral_1_over_norm_transform}
\int\limits_{\substack{{\bf z}\in S_F^1(\infty)\\1 \leq N({\bf z}) \leq B}}\frac{d{\bf z}}{N({\bf z})} = \int\limits_{N_1^{-1}([1,B])}\frac{dm_d}{N_1({\bf z})} = \int\limits_{[1,B]}\frac{1}{t}d(m_d \circ N_1^{-1}) = \Vol S_F^1(1) \log B\text.
\end{equation}
Let $A(T) := \{{\bf z} \in S_F^1(\infty) \mid 1 \leq N({\bf z}) \leq T \text{, \eqref{int_smaller_sum_error_case_1_start} or \eqref{int_smaller_sum_error_case_2_start} holds for }{\bf z}\}$. Then $A(T)$ is measurable for all $T$ and the restriction of $N$ to $A(B)$ defines a measurable function $N_2 : A(B) \to [1, B]$. For any $E \subseteq [1, B]$ with $m_1(E) = 0$, we have $N_2^{-1}(E) \subseteq N_1^{-1}(E)$ and $(m_d \circ N_1^{-1})(E) = 0$. Thus, $m_d \circ N_2^{-1}$ is absolutely continuous. With the distribution function $F(T) := (m_d \circ N_2^{-1})([1, T])$, we obtain 
\begin{equation}\label{integral_transform}
\int_{A(B)} \frac{d{\bf z}}{N({\bf z})} = \int_{N_2^{-1}([1,B])}\frac{dm_d}{N_2({\bf z})} = \int_{[1,B]}\frac{1}{t}d(m_d \circ N_2^{-1}) = \int_1^B\frac{1}{t}dF(t)\text.
\end{equation}
Integration by parts for the Stieltjes integral on the right-hand side suggests that we need to find a suitable bound for $F(T)$. Clearly,
\[F(T) = \Vol(N_2^{-1}([1,T])) = \Vol A(T)\text.\]
With $H$, $D_w$ as in \eqref{def_H}, \eqref{def_Dw}, let
\begin{align*}
A_1(T) &:= \{{\bf z} \in \R^d \mid ({\bf z} + D_w) \cap \partial S_F^1(T^{1/d}) \neq \emptyset\}\text{, and}\\
A_2(T) &:= \{{\bf z} \in \R^d \mid ({\bf z} + D_w) \cap (\overline{S_F^1(T^{1/d})} \cap H) \neq \emptyset\}\text.
\end{align*}
A similar argument to before shows that $A(T) \subseteq A_1(T) \cup A_2(T)$. We already know that $\partial S_F^1(T^{1/d}) \in \Lip(d, M_1, T^{1/d} L_1)$ and $S_F^1(T^{1/d}) \cap H \in \Lip(n, \tilde M_1, T^{1/d} \tilde L_1)$. The same holds of course for the closure. By Lemma \ref{translate_of_fixed_set}, \emph{(ii)}, we obtain
\[\Vol A_1(T) \ll T^{(d-1)/d}\text{, }\Vol A_2(T) \ll T^{(d-1)/d}\text{, for }T\geq 1\text,\]
and thus $F(T) \ll T^{(d-1)/d}$, for $T \geq 1$. Integration by parts gives
\[\int_1^B\frac{1}{t}dF(t) = F(B)/B - F(1) - \int_1^B F d\frac{1}{t} \ll B^{-1/d} + \int_1^B t^{-(1+1/d)}dt \ll 1\text.\]
With \eqref{int_smaller_sum_split_up}, \eqref{integral_1_over_norm_transform} and \eqref{integral_transform}, we obtain \eqref{int_smaller_sum_hard_part}. Together with \eqref{int_smaller_sum_easy_part} this gives \eqref{int_smaller_sum}.
\end{proof}

\begin{lemma}\label{compute_IB}
We have
\[I(B) = \frac{1}{4\cdot 6!}\left(\frac{2^r(2\pi)^sR_K}{\sqrt{|\Delta_K|}}\right)^6 \frac{1}{\prod_{v \in V'}\Norm\idb_{v}}(\log B)^6\text.\]
\end{lemma}

\begin{proof}
Let $m_n$ denote the Lebesgue measure on $\R^n$. We define the measurable function $f : (S_F^1(\infty))^6 \to \R^6$ by $f(({\bf z}_v)_{v \in V'}) = (N({\bf z}_v))_{v\in V'}$. For any cell $E := \prod_{v \in V'}(a_v, b_v]$, with $0 < a_v \leq b_v$, we have
\begin{align*}
(m_{6d} \circ f^{-1})(E) &= \prod_{v \in V'}(\Vol S_F^1(b_v^{1/d}) - \Vol S_F^1(a_v^{1/d})) = (\Vol S_F^1(1))^6 m_6(E)\text.
\end{align*}
Thus, $m_{6d}\circ f^{-1} = (\Vol S_F^1(1))^6 m_6$ on $(\R^{\geq 0})^6$. Let
\[M_\Q(B) := \{(t_v)_{v \in V'} \in \R^6 \mid t_v \geq 1\text{ for all }v\text{ and }t_{jk}t_{jl}t_{kj}^2t_{lj}^2 \leq B \text{ for all }j\}\text.\]
Then 
\begin{align*}
\int\limits_{M(B)}\prod_{v \in V'}\frac{d{\bf z}_v}{N({\bf z}_v)} &= \int\limits_{f^{-1}(M_\Q(B))}\prod_{v\in V'}\frac{1}{f({\bf z})_v}dm_{6d} = \int\limits_{M_\Q(B)}\prod_{v \in V'}\frac{1}{t_v}d(m_{6d} \circ f^{-1})\\ &= (\Vol S_F^1(1))^6 \int\limits_{M_\Q(B)}\prod_{v \in V'}\frac{1}{t_v} dm_6 = \frac{(\Vol S_F^1(1))^6}{4 \cdot 6!}(\log B)^6\text.
\end{align*}
The last integral is computed at the end of \cite{Heath-Brown1999}.
\end{proof}
We define 
\[C_0(K) := \frac{1}{4 \cdot 6!}\left(\frac{2^r (2\pi)^s R_K}{\sqrt{|\Delta_K|}}\right)^6\quad\text{ and }\quad C(K) := 3^q C_0(K)\text.\]
Then \eqref{M1_first_equality} and the previous two lemmata imply that
\[\mathcal{M}_1(B, (\idb_v)_v) = \frac{C_0(K)}{\prod_{v \in V'}\Norm\idb_v}(\log B)^6 + O(\log B)^5\text.\]
Keep in mind that $\idb_v \in \mathcal{C}$ for all $v \in V'$. With \eqref{error_term_norm_bound_3}, \eqref{M_independent_of_fundamental_system}, we obtain
\begin{equation*}
\mathcal{M}(B, (\ida_v)_v) \leq\frac{C(K)}{\prod_{v \in V'}\Norm\ida_v}(\log B)^6 + O\left(\frac{1}{\prod_{v \in V'}\Norm\ida_v}(\log B)^5\right)\text.
\end{equation*}
Let $R := \max_j\{\Norm\ida_j\}^{1/d}\prod_{v \in V'}\Norm\ida_v^{2/(3d)}$. Then $R \geq c_4 > 0$ for some constant $c_4$ depending only on $K$. This implies in particular that $\log R \ll R$. Moreover, we have $1/(c_5R^{3d}) \leq b_j$, for some constant $c_5 \geq 1$ depending only on $K$. Therefore,
\[\mathcal{M}(B, (\ida_v)_v) \geq 3^{q}\left(\prod_{v \in V'}\frac{\Norm\idb_v}{\Norm\ida_v}\right)\mathcal{M}_1(B/(c_5 R^{3d}), (\idb_v)_v)\text.\]
Whenever $B \geq e c_5 R^{3d}$, we obtain
\begin{align*}
\mathcal{M}(B, (\ida_v)_v) &\geq \frac{C(K)}{\prod_{v \in V'}\Norm\ida_v}\log(B/(c_5 R^{3d}))^6 + O\left(\frac{1}{\prod_{v \in V'}\Norm\ida_v}\log(B/(c_5 R^{3d}))^5\right)\\
&= \frac{C(K)}{\prod_{v \in V'}\Norm\ida_v}(\log B)^6 + O(\frac{R}{\prod_{v \in V'}\Norm\ida_v}(\log B)^5)\text.
\end{align*}
This result holds as well if $e \leq B < e c_5 R^{3d}$, since then the error term dominates the main term. Therefore, 
\[\mathcal{M}(B, (\ida_v)_v) = \frac{C(K)}{\prod_{v \in V'}\Norm\ida_v}(\log B)^6 + O(\frac{R}{\prod_{v \in V'}\Norm\ida_v}(\log B)^5)\text,\]
and Lemma \ref{main_lemma} follows from \eqref{only_main_term_left}.

\bibliographystyle{plain}

\begin{thebibliography}{10}

\bibitem{Baier2011}
S.~Baier and T.~D. Browning.
\newblock Inhomogeneous cubic congruences and rational points on del pezzo
  surfaces.
\newblock to appear in J. reine angew. Math.

\bibitem{Batyrev1998}
V.~V. Batyrev and Y.~Tschinkel.
\newblock Manin's conjecture for toric varieties.
\newblock {\em J. Algebraic Geom.}, 7(1):15--53, 1998.

\bibitem{Batyrev1998b}
V.~V. Batyrev and Y.~Tschinkel.
\newblock Tamagawa numbers of polarized algebraic varieties.
\newblock {\em Ast{\'e}risque}, (251):299--340, 1998.
\newblock Nombre et r{\'e}partition de points de hauteur born{\'e}e (Paris,
  1996).
  
\bibitem{Breteche1998}
R.~de~la Bret{\`e}che.
\newblock Sur le nombre de points de hauteur born{\'e}e d'une certaine surface
  cubique singuli{\`e}re.
\newblock {\em Ast{\'e}risque}, (251):51--77, 1998.
\newblock Nombre et r{\'e}partition de points de hauteur born{\'e}e (Paris,
  1996).

\bibitem{Breteche2002}
R.~de~la Bret{\`e}che.
\newblock Nombre de points de hauteur born{\'e}e sur les surfaces de del
  {P}ezzo de degr{\'e} 5.
\newblock {\em Duke Math. J.}, 113(3):421--464, 2002.

\bibitem{Breteche2008}
R.~de~la Bret{\`e}che and T.~D. Browning.
\newblock Manin's conjecture for quartic del {P}ezzo surfaces with a conic
  fibration.
\newblock {\em Duke Math. J.}, 160:1--69, 2011.

\bibitem{Breteche2007c}
R.~de~la Bret{\`e}che, T.~D. Browning, and U.~Derenthal.
\newblock On {M}anin's conjecture for a certain singular cubic surface.
\newblock {\em Ann. Sci. {\'E}cole Norm. Sup. (4)}, 40(1):1--50, 2007.

\bibitem{Breteche2004}
R.~de~la Bret{\`e}che and {\'E}.~Fouvry.
\newblock L'{\'e}clat{\'e} du plan projectif en quatre points dont deux
  conjugu{\'e}s.
\newblock {\em J. Reine Angew. Math.}, 576:63--122, 2004.

\bibitem{Breteche2007}
R.~de~la Bret{\`e}che and P.~Swinnerton-Dyer.
\newblock Fonction z{\^e}ta des hauteurs associ{\'e}e {\`a} une certaine
  surface cubique.
\newblock {\em Bull. Soc. Math. France}, 135(1):65--92, 2007.

\bibitem{Browning2009}
T.~D. Browning and U.~Derenthal.
\newblock Manin's conjecture for a cubic surface with {$D_5$} singularity.
\newblock {\em Int. Math. Res. Not.}, (14):2620--2647, 2009.

\bibitem{Breteche2012}
R.~de~la Bret{\`e}che, T.~D. Browning, and E.~Peyre.
\newblock On {M}anin's conjecture for a family of Ch{\^a}telet surfaces.
\newblock {\em Annals of Math.}, 175:297--343, 2012.

\bibitem{Chambert-Loir2002}
A.~Chambert-Loir and Y.~Tschinkel.
\newblock On the distribution of points of bounded height on equivariant
  compactifications of vector groups.
\newblock {\em Invent. Math.}, 148(2):421--452, 2002.

\bibitem{Christensen2008}
C.~Christensen and W.~Gubler.
\newblock Der relative {S}atz von {S}chanuel.
\newblock {\em Manuscripta Math.}, 126(4):505--525, 2008.

\bibitem{Colliot1980}
J.-L. Colliot-Th{\'e}l{\`e}ne and J.-J. Sansuc.
\newblock La descente sur les vari{\'e}t{\'e}s rationnelles.
\newblock In {\em Journ{\'e}es de {G}{\'e}ometrie {A}lg{\'e}brique d'{A}ngers,
  {J}uillet 1979/{A}lgebraic {G}eometry, {A}ngers, 1979}: 223--237.
  Sijthoff \& Noordhoff, Alphen aan den Rijn, 1980.

\bibitem{Colliot1987}
J.-L. Colliot-Th{\'e}l{\`e}ne and J.-J. Sansuc.
\newblock La descente sur les vari{\'e}t{\'e}s rationnelles. {II}.
\newblock {\em Duke Math. J.}, 54(2):375--492, 1987.

\bibitem{Derenthal2011}
U.~Derenthal and F.~Janda.
\newblock Gaussian rational points on a singular cubic surface.
\newblock To appear in: {\em Torsors, étale homotopy and applications to rational points -- {P}roceedings of the ICMS workshop in {E}dinburgh, 10--14 {J}anuary 2011}.
\newblock London {M}athematical {S}ociety {L}ecture {N}ote {S}eries

\bibitem{Fouvry1998}
{\'E}.~Fouvry.
\newblock Sur la hauteur des points d'une certaine surface cubique
  singuli{\`e}re.
\newblock {\em Ast{\'e}risque}, (251):31--49, 1998.
\newblock Nombre et r{\'e}partition de points de hauteur born{\'e}e (Paris,
  1996).

\bibitem{Franke1989}
J.~Franke, Y.~I. Manin, and Y.~Tschinkel.
\newblock Rational points of bounded height on {F}ano varieties.
\newblock {\em Invent. Math.}, 95(2):421--435, 1989.

\bibitem{Heath-Brown1999}
D.~R. Heath-Brown and B.~Z. Moroz.
\newblock The density of rational points on the cubic surface
  {$X_0^3=X_1X_2X_3$}.
\newblock {\em Math. Proc. Cambridge Philos. Soc.}, 125(3):385--395, 1999.

\bibitem{Lang1994}
S.~Lang.
\newblock {\em Algebraic number theory}, volume 110 of {\em Graduate Texts in
  Mathematics}.
\newblock Springer-Verlag, New York, second edition, 1994.

\bibitem{Boudec2011a}
P.~{Le Boudec}.
\newblock Manin's conjecture for a cubic surface with $2A_2+A_1$ singularity
  type.
\newblock To appear in {\em Math. Proc. Cambridge Philos. Soc.}, 2012.

\bibitem{Masser2006}
D.~Masser and J.~D. Vaaler.
\newblock Counting algebraic numbers with large height. {II}.
\newblock {\em Trans. Amer. Math. Soc.}, 359(1):427--445, 2006.

\bibitem{Peyre1995}
E.~Peyre.
\newblock Hauteurs et mesures de {T}amagawa sur les vari{\'e}t{\'e}s de {F}ano.
\newblock {\em Duke Math. J.}, 79(1):101--218, 1995.

\bibitem{Salberger1998}
P.~Salberger.
\newblock Tamagawa measures on universal torsors and points of bounded height
  on {F}ano varieties.
\newblock {\em Ast{\'e}risque}, (251):91--258, 1998.
\newblock Nombre et r{\'e}partition de points de hauteur born{\'e}e (Paris,
  1996).

\bibitem{Schanuel1979}
S.~H. Schanuel.
\newblock Heights in number fields.
\newblock {\em Bull. Soc. Math. France}, 107(4):433--449, 1979.

\bibitem{Widmer2010}
M.~Widmer.
\newblock Counting primitive points of bounded height.
\newblock {\em Trans. Amer. Math. Soc.}, 362:4793--4829, 2010.

\end{thebibliography}

\end{document}